\documentclass[11pt]{amsart}
\usepackage{latexsym,amscd,amssymb, graphicx, shuffle, float}  
\usepackage[margin=1in]{geometry}

\newtheorem{theorem}{Theorem}
\newtheorem{proposition}[theorem]{Proposition}
\newtheorem{corollary}[theorem]{Corollary}
\newtheorem{lemma}[theorem]{Lemma}

\newtheorem{example}[theorem]{Example}

\newtheorem{defn}[theorem]{Definition}
\theoremstyle{definition}

\newcommand{\Cat}{{{\sf Cat}}}

\newcommand{\Ass}{{{\sf Ass}}}
\newcommand{\Asshat}{{\widehat{\sf Ass}}}
\newcommand{\dl}{{\mathrm {dl}}}

\newcommand{\symm}{{\mathfrak{S}}}

\newcommand{\ZZ}{{\mathbb {Z}}}

\newcommand{\PP}{{\mathbb {P}}}

\newcommand{\QQ}{{\mathbb{Q}}}

\newcommand{\OG}{{\mathcal{OG}}}
\newcommand{\F}{{\mathcal{F}}}
\newcommand{\D}{{\mathcal{D}}}

\begin{document}

\title[Alexander duality and rational associahedra]
{Alexander duality and rational associahedra}

\author{Brendon Rhoades}
\address
{Deptartment of Mathematics \newline \indent
University of California, San Diego \newline \indent
La Jolla, CA, 92093-0112, USA}
\email{bprhoades@math.ucsd.edu}

\begin{abstract}
A recent pair of papers of Armstrong, Loehr, and Warrington \cite{ALW} and Armstrong, Williams, and the author
\cite{ARW}
initiated the systematic study of  {\em rational Catalan combinatorics} which is a generalization of
Fu\ss-Catalan combinatorics (which is in turn a generalization of classical Catalan combinatorics).  
The latter paper gave two possible models
for a rational analog of the associahedron which attach  simplicial complexes to any pair
of coprime positive integers $a < b$.  
These complexes coincide up to the Fu\ss-Catalan level of generality, but at the rational
level of generality one may be
a strict subcomplex of the other.
Verifying Conjecture 4.7 of \cite{ARW},
we prove 
that these complexes agree up to homotopy and, in fact, that one complex collapses onto the other.  
This reconciles the two competing models for rational associahedra.
As
a corollary, we get that the involution $(a < b) \longleftrightarrow (b-a < b)$ on pairs of 
coprime positive integers manifests itself topologically as Alexander duality of rational associahedra.
This collapsing and Alexander duality are new features of rational Catalan combinatorics which
are invisible at the Fu\ss-Catalan level of generality.
\end{abstract}

\keywords{simplicial complex, lattice path, associahedron, collapsing}
\maketitle

\begin{figure}[h]
\centering
\includegraphics[scale = 0.7]{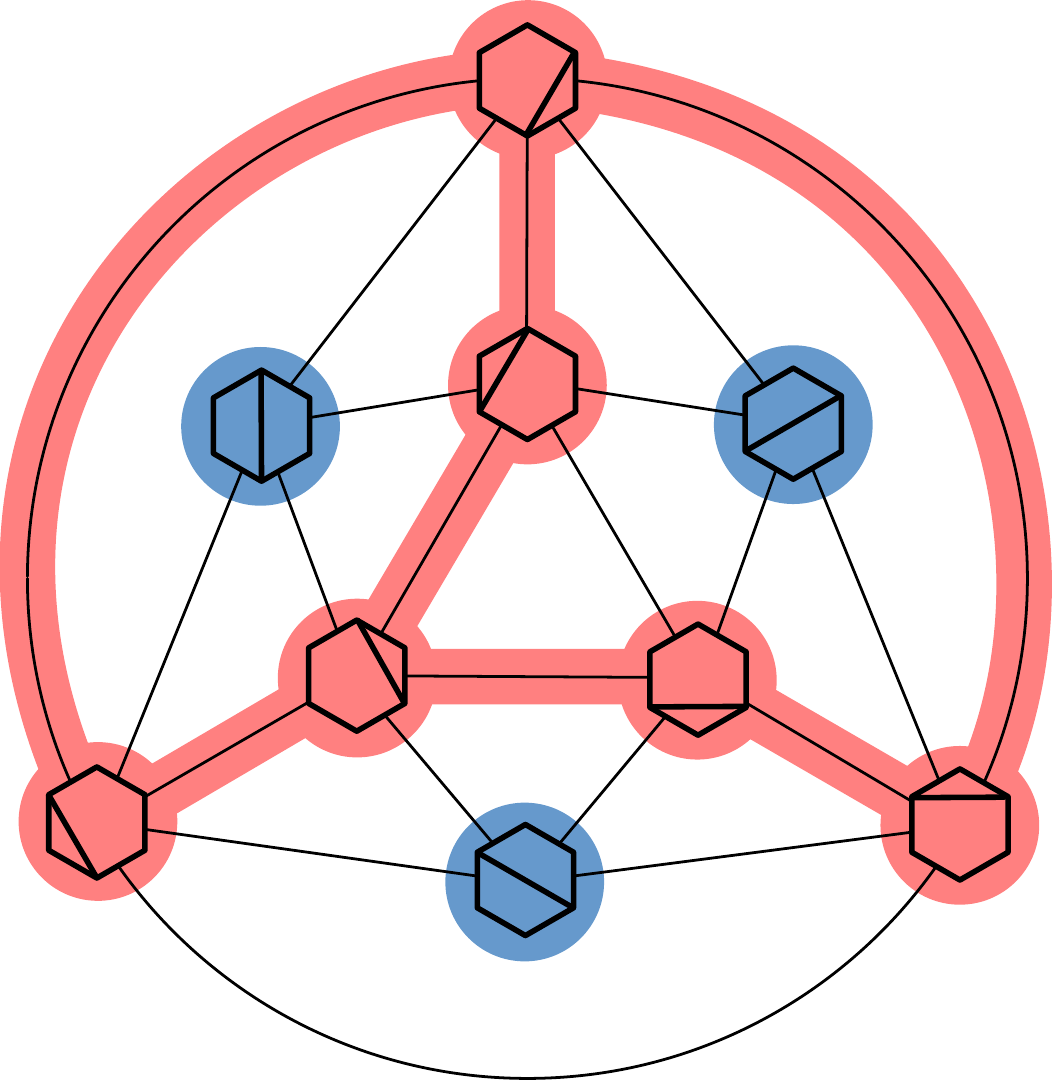}
\caption{The complexes $\Ass(3,5)$ (red) and $\Ass(2,5)$ (blue) are Alexander dual within the $2$-dimensional
sphere $\Ass(4,5)$ (black).}
\label{alexanderdual}
\end{figure}

\section{Introduction}
\label{Introduction}

For $n > 0$, let $\PP_{n+2}$ denote the regular $(n+2)$-gon.  The
\textbf{\textit{(classical) associahedron}}
is the  $(n-2)$-dimensional simplicial sphere whose faces are the dissections
of $\PP_{n+2}$.  
The polytopal dual to this complex was introduced by Stasheff
\cite{Stasheff} to study  nonassociative binary operations arising in algebraic topology.

Given a Fu$\ss$ parameter $k \in \ZZ_{> 0}$, a  dissection of 
$\PP_{kn+2}$ is called \textbf{\textit{$k$-divisible}} if the number of vertices in each sub-polygon
is congruent to $2$ modulo $k$.  Przytycki and Sikora \cite{PS} proved that the number
of $k$-divisible dissections of $\PP_{kn+2}$ with $i$ diagonals equals
$\frac{1}{n} {kn+i+1 \choose i} {n \choose i+1}$.  Using this result and 
the Fomin-Zelevinsky cluster complexes \cite{FZ} as motivation, Tzanaki 
\cite{Tzanaki} 
studied the  \textbf{\textit{generalized associahedron}}
whose faces are the $k$-divisible dissections of $\PP_{kn+2}$. 
She proved that this generalization of the associahedron is shellable.  
 As a corollary, 
the generalized associahedron is homotopy equivalent to a wedge of $k$ spheres, all
 of dimension $n-2$.

Rational Catalan combinatorics is a further generalization of Fu\ss-Catalan combinatorics which
depends on a pair of coprime positive integers $a < b$, thought of as corresponding to the rational
number $\frac{b-a}{a} \in \QQ_{> 0}$.  The choice $(a, b) = (n, n+1)$ recovers classical Catalan theory
and the choice $(a, b) = (n, kn+1)$ recovers Fu\ss-Catalan theory.  

While 
some of the objects considered in rational Catalan theory date back at least to the 1950s 
(see \cite{Bizley}),
its systematic study was initiated only recently.  In particular, 
Armstrong, Williams, and the author \cite{ARW} defined and studied rational generalizations
of Dyck paths, noncrossing and nonnesting partitions of $\{1, 2, \dots, n \}$, noncrossing perfect matchings
on $\{1, 2, \dots, 2n \}$, and the associahedron.  In a companion paper, Armstrong, Loehr, and
Warrington \cite{ALW} defined rational parking functions and studied generalizations
of the statistics `area', `dinv', and `iDes' to this setting.  

The motivation for this combinatorial 
program was Gordon's use of rational Cherednik algebras to give a generalization of the 
diagonal coinvariant ring to any reflection group $W$ \cite{Gordon}.  
The favorable representation theoretic
properties of these algebras at parameter $\frac{b}{h}$, where $h$ is the Coxeter number of $W$ and
$h < b$ are coprime, suggested the problem of studying the combinatorics of this 
`rational' case when $W = \symm_a$.  Generalizing rational Catalan combinatorics beyond
type A is almost entirely an open problem.

This paper focuses on the rational analog of the associahedron, which gives a generalization
of the classical and generalized associahedra.  We prove a conjecture of Armstrong, Williams,
and the author \cite[Conjecture 4.7]{ARW}, obtaining the conjecture \cite[Proposition 4.8]{ARW} as a corollary
(the proof of \cite[Proposition 4.8]{ARW} was given 
in \cite{ARW} as an easy consequence of
\cite[Conjecture 4.7]{ARW}).  By verifying these
conjectures, we will uncover some genuinely new features of rational Catalan combinatorics
which are invisible at the classical and Fu\ss-Catalan levels of generality.  This gives evidence that the rational
level of generality is of combinatorial interest.

Given coprime positive integers $a < b$, Armstrong, Williams, and the author \cite{ARW} defined
two simplicial complexes $\Ass(a,b)$ and $\Asshat(a,b)$ called `rational associahedra'
whose faces are certain dissections of $\PP_{b+1}$.
When $a = n$ and $b = kn+1$, the complexes
$\Ass(n, kn+1)$ and $\Asshat(n, kn+1)$ coincide and both equal the generalized associahedron
of $k$-divisible dissections of $\PP_{kn+2}$.  In general, we have that $\Ass(a,b)$ is a subcomplex of
$\Asshat(a,b)$, and the case of $(a, b) = (3, 5)$ shows that this inclusion can be strict.  The complexes
$\Asshat(3,5)$ and $\Ass(3,5)$ are shown on the left and right of Figure~\ref{retract}, respectively.

At the rational level of generality, many of the nice features of associahedra diverge between the two 
constructions $\Ass(a,b)$ and $\Asshat(a,b)$.  The complex $\Asshat(a,b)$ has a simpler definition
which is more intrinsically related to polygon dissections and 
is closed under the dihedral symmetries of $\PP_{b+1}$.  However, the complex
$\Asshat(a,b)$ is not in general pure  and there
does not appear to be a nice formula for the entires of its $f$- and $h$-vectors. 
 The complex $\Ass(a,b)$ 
has a more complicated definition involving lattice paths and
does not carry an
action of the symmetry group of $\PP_{b+1}$.  
However, the complex $\Ass(a,b)$ is
pure and shellable, and there are  nice product formulas
for the $f$- and $h$-vector entries of $\Ass(a,b)$ given by rational analogs of the Kirkman and 
Narayana numbers. 

Needless to say, having two different models for rational associahedra is unfortunate and the different
advantageous features of $\Ass(a,b)$ and $\Asshat(a,b)$ make it difficult to choose which model is ``correct".
We ameliorate this problem by proving 
that these complexes are equivalent up to combinatorial homotopy.

\begin{theorem}
\label{main} (\cite[Conjecture 4.7]{ARW})
Let $a < b$ be coprime positive integers.  The simplicial complex $\Asshat(a,b)$ collapses onto the simplicial
complex $\Ass(a,b)$.
\end{theorem}

Theorem~\ref{main} is trivial  up to the Fu$\ss$ level of generality because we have the
equality of complexes
$\Asshat(n, kn+1) = \Ass(n, kn+1)$; the disagreement between the complexes 
$\Asshat(a,b)$ and $\Ass(a,b)$ and the resulting nontrivial
collapsing appears only at the rational level of generality.  
In the case $(a, b) = (3, 5)$, Theorem~\ref{main}
can be proven by observing that in Figure~\ref{retract}, the complex on the right can be obtained from the complex
on the left by collapsing the top and bottom triangles.

To prove Theorem~\ref{main}, we will describe an explicit collapse of $\Asshat(a,b)$ onto 
$\Ass(a,b)$.  The proof of Theorem~\ref{main} will be given at the end of the paper after a series
of intermediate results.
The idea is
to identify the obstructions which prevent a face of $\Asshat(a,b)$ from being a face of the subcomplex
$\Ass(a,b)$.  These 
obstructions will be local in nature and can be encoded as the edges of an `obstruction graph'.
Using a well chosen total order on these edges, we can use a sequence of collapses
to eliminate these obstructions and prove Theorem~\ref{main}.

As was noted in \cite[Proposition 4.8]{ARW}, Theorem~\ref{main} has a corollary which gives a topological
relationship between different rational associahedra.  Let $S$ be a  sphere.  A topological subspace $X$
of $S$ is said to be \textbf{\textit{Alexander dual}} to the complement $S - X$.  Generalizing slightly, we also
say that two topological subspaces $X$ and $Y$ of $S$ are Alexander dual to one another if
$S-X$ deformation retracts onto $Y$ and $S-Y$ deformation retracts onto $X$.

For any $a < b$ coprime, 
the faces of the complexes $\Ass(a,b)$ and $\Asshat(a,b)$ are given by dissections of 
$\PP_{b+1}$, so either of these complexes embeds in the $(b-3)$-dimensional 
simplicial sphere given by the classical associahedron $\Ass(b-1,b)$.
We have the following  involution on increasing
pairs of coprime positive integers:
\begin{equation}
\label{integer-duality}
a < b \longleftrightarrow b-a < b.
\end{equation}
At the level of rational associahedra, this ``categorifies" to Alexander duality.

\begin{corollary} (see \cite[Proposition 4.8]{ARW})
\label{alexander}
Let $a < b$ be coprime positive integers.  
The simplicial complexes $\Asshat(a,b)$ and $\Asshat(b-a,b)$ are Alexander dual within the sphere
$\Ass(b-1,b)$.  
The simplicial complexes $\Ass(a,b)$ and $\Ass(b-a,b)$ are also Alexander
dual within the sphere $\Ass(b-1,b)$.
\end{corollary}

Figure~\ref{alexanderdual} shows the complexes $\Ass(3,5)$ (in red) and $\Ass(2,5)$ (in blue)
as Alexander duals within the $2$-sphere $\Ass(4,5)$ (in black).  Since the `Catalan' 
pairs $\{(n, n+1)\}$ and the `Fu\ss-Catalan' pairs $\{(n, kn+1)\}$ are not closed under 
the above involution, Corollary~\ref{alexander} is another genuinely new feature of rational Catalan
theory.
Corollary~\ref{alexander} was proven for the $\Asshat$ complexes in \cite{ARW} using a relatively elementary
argument.  By showing that $\Ass(a,b)$ is shellable and computing its $h$-vector, it is also shown in
\cite{ARW} that $\Ass(a,b)$ is homotopy equivalent to a wedge of $\frac{1}{b}{b \choose a}$ spheres,
all of dimension $a-2$.  This implies that the homology groups of $\Ass(a,b)$ and $\Ass(b-a,b)$ have ranks
predicted by Corollary~\ref{alexander}.

Along the way of proving Theorem~\ref{main} and
Corollary~\ref{alexander}, we will also give a structural result on the complexes
$\Ass(a,b)$.  

\begin{proposition}
\label{flag}
The simplicial complex $\Ass(a,b)$ is flag.
\end{proposition}

It follows from its definition that $\Asshat(a,b)$ is also a flag simplicial complex.

The remainder of the paper is organized as follows.  In \textbf{Section~\ref{Background}} we give background
definitions related to simplicial complexes and recall the key combinatorial tool (Lemma~\ref{cone-vertex-lemma})
we will use to perform the collapsing in Theorem~\ref{main}.  
In \textbf{Section~\ref{Rational}} we review the definitions of $\Ass(a,b)$ and $\Asshat(a,b)$
from \cite{ARW}.
In \textbf{Section~\ref{Flag complexes}} we begin our analysis of the difference
$\Asshat(a,b) - \Ass(a,b)$ and prove Proposition~\ref{flag}.  
In \textbf{Section~\ref{Obstruction graphs}} we introduce the notion of an obstruction graph, which 
will serve as our gadget for keeping track of the obstructions which prevent a face of $\Asshat(a,b)$ from being
a face of $\Ass(a,b)$.
In \textbf{Section~\ref{Collapsing}} we prove Theorem~\ref{main}.  For the sake of completeness,
we also recall from \cite{ARW} how Corollary~\ref{alexander} can be deduced from 
Theorem~\ref{main}.  We make closing remarks in \textbf{Section~\ref{Open}}.

\section{Background on simplicial complexes}
\label{Background}

Let $E$ be a finite set.  A \textbf{\textit{simplicial complex}} $\Delta$ (on the \textbf{\textit{ground set}} $E$)
is a collection of subsets of $E$ such that if $F \in \Delta$ and $F' \subseteq F$, we have that 
$F' \in \Delta$.  Elements of $\Delta$ are called \textbf{\textit{faces}} and elements of the ground set
$E$ are called \textbf{\textit{vertices}}.   
A \textbf{\textit{facet}} of $\Delta$ is a maximal face of $\Delta$.  The \textbf{\textit{dimension}} $\dim(F)$ of a  face
$F \in \Delta$ is given by $\dim(F) := |F| - 1$.  
The \textbf{\textit{dimension}} $\dim(\Delta)$ of the complex $\Delta$ is the maximum
$\max \{ \dim(F) \,:\, F \in \Delta \}$ of the dimensions of its faces.
The complex $\Delta$ is called
\textbf{\textit{pure}} if all of the facets of $\Delta$ have the same dimension.

Let $\Delta$ be a simplicial complex with $\dim(\Delta) = d$.  The \textbf{\textit{$f$-vector}}
$f(\Delta) = (f_{-1}, f_0, \dots, f_d)$ is the vector obtained by letting $f_i$ be the number 
of $i$-dimensional faces of $\Delta$.  The \textbf{\textit{$h$-vector}} is the vector
$h(\Delta) = (h_0, \dots, h_d)$ defined by
$h_k = \sum_{i = 0}^k (-1)^{k-i} {d - i \choose k - i} f_{i-1}$.

Let $\Delta$ be a simplicial complex on the ground set $E$ and let $F$ be a subset of $E$ with
$|F| \geq 2$.  The set 
$F$ is called an \textbf{\textit{empty face}} if $F$ is not a face of $\Delta$, but every two-element subset
$\{v, v'\} \subseteq F$ of $F$ is a face of $\Delta$.  The complex $\Delta$ is called \textbf{\textit{flag}}
if $\Delta$ does not contain any empty faces.  If $\Delta$ is flag, then $\Delta$ is determined by
its $1$-skeleton $\{ F \in \Delta \,:\, |F| \leq 2 \}$.

Let $\Delta$ be a simplicial complex, let $F$ be a facet of $\Delta$, and let $F' \subset F$ be a face satisfying
$\dim(F') = \dim(F) - 1$.  The pair $(F, F')$ is called \textbf{\textit{free}} if $F$ is the unique face of $\Delta$ 
satisfying $F' \subset F$.
When the pair $(F, F')$ is free, the set of faces
$\Delta - \{F, F'\}$ is a subcomplex of $\Delta$ and we say that 
$\Delta - \{F, F'\}$ is obtained from $\Delta$ by an \textbf{\textit{elementary collapse}} (along the free pair 
$(F, F')$).  Topologically, this corresponds to the deformation
retraction given by crushing the facet $F$ through $F'$.  Given an arbitrary subcomplex $\Delta'$ of $\Delta$,
we say that $\Delta$ \textbf{\textit{collapses onto}} $\Delta'$ if there exists a sequence of subcomplexes
$\Delta = \Delta_1 \supset \Delta_2 \supset \dots \supset \Delta_n = \Delta'$ such that
$\Delta_{i+1}$ is obtained from $\Delta_i$ by an elementary collapse for $1 \leq i \leq n-1$.
Collapses were introduced in 1938 by Whitehead \cite{Whitehead} and give a combinatorial model
for certain deformation retractions.  Collapses are transitive in the sense that if $\Delta$ collapses onto
$\Delta'$ and $\Delta'$ collapses onto $\Delta''$, then $\Delta$ collapses onto $\Delta''$.

Let $\Delta$ be a simplicial complex on the ground set $E$ and let $F$ be a face of $\Delta$.  
We define $\Delta(F) := \{ F' \in \Delta \,:\, F \subseteq F' \}$ to be the set of faces of $\Delta$ containing $F$.  
We say that a ground set element
$c \in E - F$ which is not in $F$ is a \textbf{\textit{cone vertex for $F$ in $\Delta$}} 
if for any $F' \in \Delta(F)$ we have that $F' \cup \{c\} \in \Delta$.

If $\Delta$ is a simplicial complex and $F$ is a face of $\Delta$, define the 
\textbf{\textit{deletion}} $\dl_{\Delta}(F)$ to be the subcomplex of $\Delta$ given by
$\dl_{\Delta}(F) := \{ F' \in \Delta \,:\, F \nsubseteq F' \} = \Delta - \Delta(F)$.  More generally,
if $\Sigma$ is any set of faces of $\Delta$, define
$\dl_{\Delta}(\Sigma) := \{ F' \in \Delta \,:\, \text{$F \nsubseteq F'$ for all $F \in \Sigma$} \}
= \bigcap_{F \in \Sigma} \dl_{\Delta}(F)$.  To prove that complexes collapse onto deletions,
we will make repeated use of the following basic lemma.

\begin{lemma}
\label{cone-vertex-lemma}
Let $\Delta$ be a simplicial complex on the ground set $E$ and let $F$ be a face of $\Delta$.  Suppose that 
there exists a cone vertex $c \in E - F$ for $F$ in $\Delta$.  Then the complex $\Delta$ collapses onto the 
deletion $\dl_{\Delta}(F)$.
\end{lemma}

\begin{proof}
Since $c$ is a cone vertex for $F$ in $\Delta$, there is a fixed-point free involution
$\phi: \Delta(F) \rightarrow \Delta(F)$ defined by
\begin{equation}
\phi(F') := \begin{cases} F' \cup \{c\} & \text{if $c \notin F'$,} \\
F' - \{c\} & \text{if $c \in F'$.} \end{cases}
\end{equation}
The map $\phi$ partitions $\Delta(F)$ into pairs $\{F, \phi(F)\}$.  We let $\leq$ be any total order on these
pairs such that whenever $F \subset \phi(F)$ and $F' \subset \phi(F')$ with $F \subseteq F'$, 
we have that $\{F, \phi(F)\} \leq \{F', \phi(F')\}$.  If we decompose 
$\Delta(F)$ as $\Delta(F) = \biguplus_{i = 1}^n \{F_i, \phi(F_i)\}$ with 
$\{F_1, \phi(F_1)\} > \dots > \{F_n, \phi(F_n) \}$, we have
that 
$\Delta - \{ F_1, \phi(F_1), F_2, \phi(F_2), \dots, F_i, \phi(F_i) \}$ is obtained from
$\Delta - \{ F_1, \phi(F_1), F_2, \phi(F_2), \dots, F_{i-1}, \phi(F_{i-1}) \}$ by an elementary collapse
along the free pair $(F_i, \phi(F_i))$ for all $1 \leq i \leq n$.  It follows that $\Delta$ collapses
onto $\dl_{\Delta}(F)$.
\end{proof}

\section{Rational associahedra}
\label{Rational}

Let $a < b$ be coprime positive integers.   We label the extremal
boundary \textbf{\textit{points}} of $\PP_{b+1}$ 
clockwise with $0, 1, 2, \dots, b$ (see Figure~\ref{dyckfacet}).
We call these ``points" rather than ``vertices" so that we do not get confused with the ground sets of 
our simplicial complexes.  
A \textbf{\textit{diagonal}} of $\PP_{b+1}$  is a line segment other than a side which connects two boundary 
points.  We use the shorthand $ij$ to refer to the diagonal connecting the boundary points $i$ and $j$
which satisfy $i < j$.
Two diagonals $ij$ and $km$ of $\PP_{b+1}$ are said to 
\textbf{\textit{cross}} if either $i < k < j < m$ or $k < i < m < j$.  

Following \cite{ARW}, we define a set
$S(a,b)$ of positive integers as follows:
\begin{equation}
S(a,b) := \left\{ \left\lfloor \frac{ib}{a} \right\rfloor \,:\, i = 1, 2, \dots, a-1 \right\}.
\end{equation}
Let $d$ be a diagonal of $\PP_{b+1}$ which separates $i$ boundary points from $j$ boundary
points, where $1 \leq i, j \leq b-2$.  The diagonal $d$ is said to be 
\textbf{\textit{$a,b$-admissible}} if $i, j \in S(a,b)$.

\begin{defn} \cite{ARW}
Let $a < b$ be coprime positive integers.  The complex $\Asshat(a,b)$ is the simplicial complex on the ground
set of $a,b$-admissible diagonals in $\PP_{b+1}$ whose faces are mutually noncrossing collections of 
$a,b$-admissible diagonals.
\end{defn}

\begin{figure}
\centering
\includegraphics[scale = 0.4]{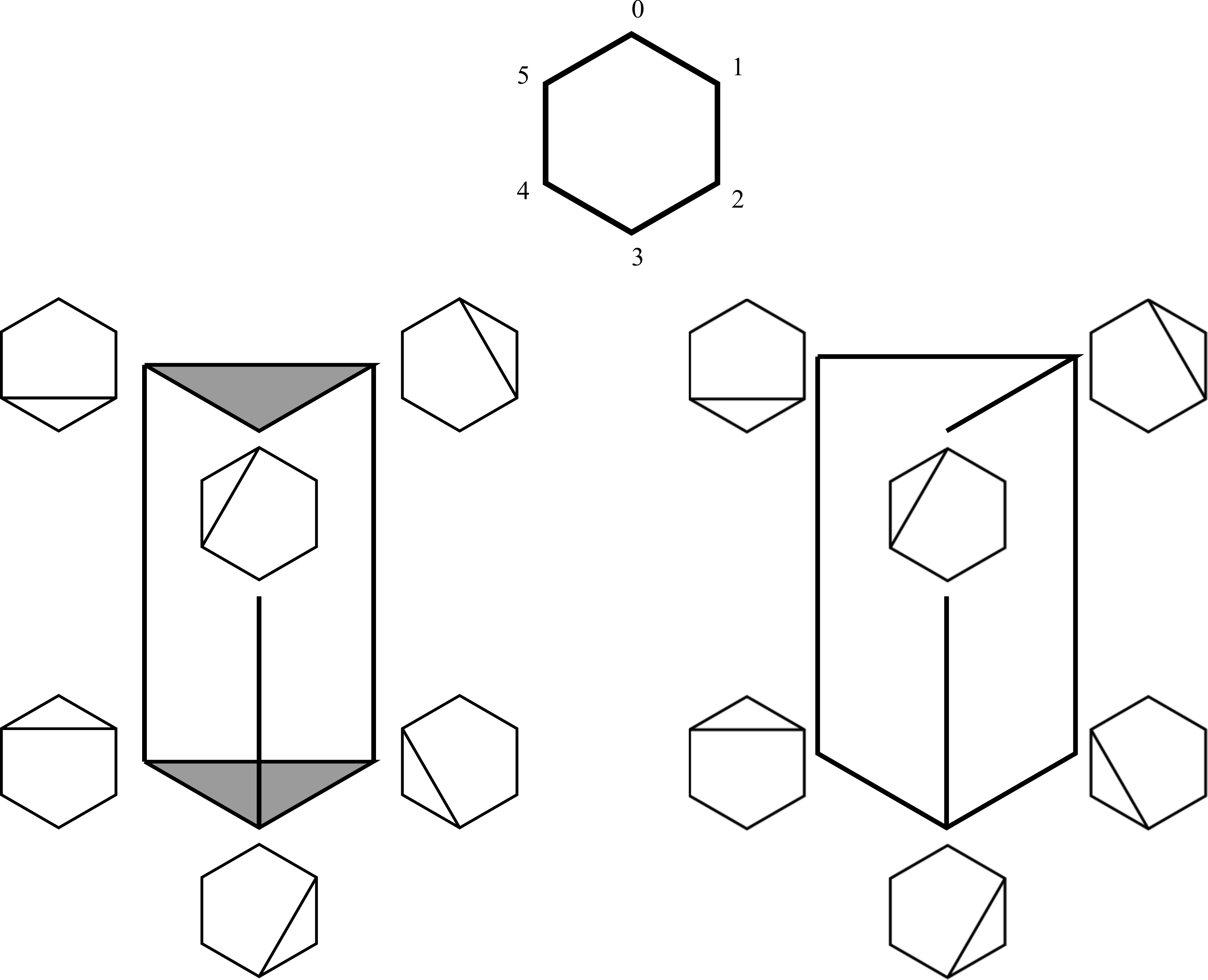}
\caption{The rational associahedra $\Asshat(3,5)$ and $\Ass(3,5)$.}
\label{retract}
\end{figure}

The complex $\Asshat(3,5)$ is shown on the left of Figure~\ref{retract}.  In general, the complex
$\Asshat(a,b)$ carries an action of the symmetry group of $\PP_{b+1}$.  Figure~\ref{retract} shows
that $\Asshat(a,b)$ need not be pure in general.
When $b = ka+1$ for $k \in \ZZ_{>0}$, we have
that $S(a, ka+1) = \{k, 2k, \dots, (a-1)k \}$ and $\Asshat(a, ka+1)$ is the complex of
$k$-divisible dissections of $\PP_{ka+2}$.

The definition of the complex $\Ass(a,b)$ is 
more involved and uses lattice paths.  An \textbf{\textit{$a,b$-Dyck path}} (or just a 
\textbf{\textit{Dyck path}} if $a$ and $b$ are clear from context) is a lattice path in $\ZZ^2$ consisting of
north and east steps which starts at $(0, 0)$, ends at $(b, a)$, and stays above the line 
$y = \frac{a}{b} x$.  A \textbf{\textit{vertical (horizontal) run}} of a Dyck path is a maximal contiguous 
sequence of east (north) steps.  A \textbf{\textit{valley}} of a Dyck path is a lattice point on the path which is
immediately preceded by an east step and immediately succeeded by a north step.  
We will sometimes write a Dyck path as a word 
of length $a + b$
in the alphabet $\{N, E\}$ consisting of a north and an
east step.

It will occasionally be convenient to think of $a,b$-Dyck paths in terms of partitions.  A 
\textbf{\textit{partition}} (of length $a$) is a sequence $\lambda = (\lambda_1 \geq \dots \geq \lambda_a)$
of weakly decreasing nonnegative integers.  We identify $\lambda$ with its
\textbf{\textit{Ferrers diagram}} which consists of $\lambda_i$ left justified boxes in row $i$.  
We let $\subseteq$ denote the partial ordering of \textbf{\textit{Young's Lattice}} given
by containment of Ferrers diagrams, so $\lambda \subseteq \mu$ if and only if $\lambda_i \leq \mu_i$
for all $i$.  An $a,b$-Dyck path $D$ cuts out a partition $\lambda(D)$ whose 
Ferrers diagram lies to the northwest of $D$.

Let $D$ be an $a,b$-Dyck path and let $P = (i, j)$ be a lattice point on $D$ other than $(0, 0)$ which lies
at the bottom of a north step of $D$.  The \textbf{\textit{laser}} 
$\ell(P) = \ell(i, j)$ of $P$ is the unique line segment of slope
$\frac{a}{b}$ whose southwest endpoint is $P$, whose northeast endpoint lies on $D$, and whose 
interior does not intersect $D$.  The northeast endpoint of $\ell(P)$ lies in the interior of an east step of $D$.
If the right endpoint of this east step has coordinates $(k, m)$, we define 
$d(P) := ik$ to be the \textbf{\textit{diagonal}} of $P$, viewed as 
a diagonal in $\PP_{b+1}$.  Since the laser $\ell(P)$ has slope $\frac{a}{b}$, the diagonal $d(P)$ is
$a,b$-admissible.  Moreover, if $P$ and $Q$ are two
bottom lattice points of north steps in $D$, we have that $d(P)$ and $d(Q)$ do not cross.  We let $\F(D)$ be the 
set
\begin{equation}
\F(D) := \{ d(P) \,:\, \text{$P$ is the lattice point at the bottom of a north step of $D$} \}.
\end{equation}
We will occasionally use the more elaborate notation $\ell_D(P)$ and $d_D(P)$ to emphasize the role
of the Dyck path $D$.
The sets $\F(D)$ as $D$ varies over  Dyck paths form the facets of $\Ass(a,b)$.

\begin{defn} \cite{ARW}
Let $a < b$ be coprime positive integers.  The complex $\Ass(a,b)$ is the simplicial complex on the 
ground set of $a,b$-admissible diagonals of $\PP_{b+1}$ whose facets are
\begin{equation}
\{ \F(D) \,:\, \text{$D$ is an $a,b$-Dyck path} \}.
\end{equation}
\end{defn}

\begin{example}
A $5,8$-Dyck path $D$ is shown in Figure~\ref{dyckfacet}.
The line $y = \frac{5}{8} x$ is dotted.
As a word in the alphabet $\{N, E\}$, the path $D$ is given by
$NNENNEEENEEEE = N^2EN^2E^3NE^4$.  The partition $\lambda(D)$ is given by
$\lambda(D) = (4,1,1,0,0)$.  The valleys of $D$ are the lattice points $(1, 2)$ and $(4, 4)$. 
Lasers have been fired from all possible lattice points of $D$.
The lasers are shown in red.  
The corresponding facet $\F(D)$ of $\Ass(5,8)$ is shown on the right as a dissection of $\PP_9$.
\end{example}

\begin{figure}
\centering
\includegraphics[scale = 0.6]{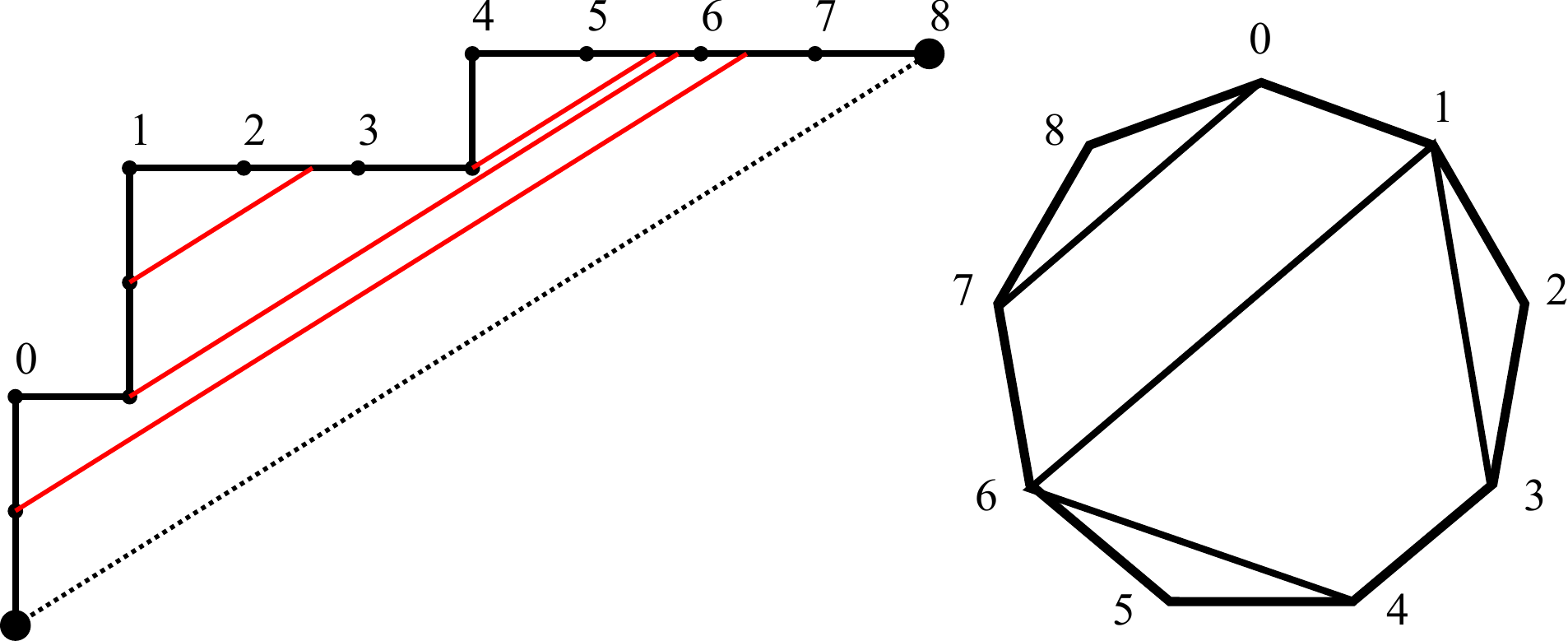}
\caption{A $5,8$-Dyck path and its facet in $\Ass(5,8)$.}
\label{dyckfacet}
\end{figure}

Since the diagonals in the facets $\F(D)$ are automatically pairwise noncrossing, we have that 
$\Ass(a,b)$ is a subcomplex of $\Asshat(a,b)$.  The complex $\Ass(a,b)$ is pure of dimension
$a-2$ and shellable.   The number of facets of $\Ass(a,b)$ is the \textbf{\textit{rational Catalan number}}
$\frac{1}{a+b}{a+b \choose a,b}$.  More generally, the number 
$f_{i-2}$ of $(i-2)$-dimensional 
faces of $\Ass(a,b)$ is given by the \textbf{\textit{rational Kirkman number}}
$\frac{1}{a} {a \choose i}{b+i-1 \choose i-1}$. 
The $h$-vector entry $h_{i-2}$ of $\Ass(a,b)$ is given by the \textbf{\textit{rational Narayana number}}
$\frac{1}{a} {a \choose i}{b-1 \choose i-1}$.
 In contrast, there is no known formula for the 
number of faces of $\Asshat(a,b)$.  
The complex $\Ass(3,5)$ is shown on the right of Figure~\ref{retract}.  

It is our aim to show that 
$\Asshat(a,b)$ collapses onto its subcomplex $\Ass(a,b)$.  At this point, it is not even clear that 
the difference $\Asshat(a,b) - \Ass(a,b)$ contains an even number of faces.

\section{Flag complexes}
\label{Flag complexes}

Let $a < b$ be coprime and let $F$ be a face of $\Asshat(a,b)$.  In order to prove that $\Asshat(a,b)$ collapses
onto $\Ass(a,b)$, we will need to develop a better understanding of when $F$ is contained in the subcomplex
$\Ass(a,b)$.  It will turn out that the obstructions preventing $F$ from being a face of $\Ass(a,b)$ can be encoded
in a graph which will give an inductive structure allowing us to apply Lemma~\ref{cone-vertex-lemma} repeatedly
to prove Theorem~\ref{main}.  The main purpose of this section is to show that 
such a graph exists by proving Proposition~\ref{flag}.

Given a face $F \in \Ass(a,b)$, there may be many Dyck paths $D$ such that $F$ is
contained in the facet $\F(D)$.  However,
it will be useful to associate a `standard' Dyck path $\D(F)$ to $F$ such 
that $F$ is contained in the facet $\F(\D(F))$.  The Dyck path $\D(F)$ will be characterized by the facts
that 
\begin{itemize}
\item
every valley of $\D(F)$ fires a laser corresponding to a diagonal in $F$, and 
\item $F \subseteq \F(\D(F))$.
\end{itemize}

\begin{lemma}
\label{valley-lemma}
Let $F$ be a nonempty
face of $\Ass(a,b)$.  There exists a unique Dyck path $\D(F)$ whose vertical runs are at the 
$x$-coordinates 
\begin{equation*}
\{ i \,:\, \text{there exists a diagonal of the form $ij$ in $F$} \}
\end{equation*}
such that the valley at the bottom of the  vertical run on the line $x = i$ for $i > 0$
has laser diagonal $i j_0$, where
$j_0$ is maximal such that $i  j_0$ is a diagonal in $F$.
\end{lemma}

\begin{proof}
Suppose that we can construct a Dyck path $\D(F)$ which satisfies the conditions of  
the lemma.  The diagonals in $F$ determine the $x$-coordinates of the vertical runs of 
$\D(F)$.  Since $a < b$ and our lasers have slope $\frac{a}{b}$, different lasers give rise to different
diagonals.  Therefore, the condition on the valley lasers implies that $F$ determines the valleys
of $\D(F)$.  This means that $F$ determines the path $\D(F)$ and $\D(F)$ is unique when it exists.  
We will show that a path $\D(F)$ satisfying the 
conditions of the lemma exists. 

To see that such a path $\D(F)$ exists, let $D$ be any Dyck path such that $F$ is a face of $\F(D)$ (we know
that such a path $D$ exists because $F$ is a face of $\Ass(a,b)$).  
If $D$ does not satisfy the condition  of the lemma, 
we will prove that there exists another Dyck path $D'$ with 
$\lambda(D') \subset \lambda(D)$
such that $F$ is also a face of $\F(D')$.  Existence then follows by induction.  Our construction
of the path $D'$ is similar to a construction used in \cite{ARW} to prove that 
 lexicographic order on the partitions $\lambda(D)$ induces a shelling order on the facets
$\F(D)$.

Suppose that $D$ does not satisfy the condition of the lemma.
Then there exists a valley $P$ of $D$ such that 
$d(P)$ is not a diagonal in $F$.  Let $N^{i_1}, \dots, N^{i_r}$ denote the vertical runs of $D$, read from east
to west.  There exists $1 < s \leq r$ such that $P$ occurs at the bottom of the vertical run $N^{i_s}$.
There are two cases to consider.

{\bf Case 1:}  {\it There are no lattice points $Q$ on the vertical runs $N^{i_1}, \dots, N^{i_{s-1}}$ such 
that the laser $\ell(Q)$ hits $D$ east of $P$.}

In this case, let $D'$ be obtained from $D$ by increasing the length of $N^{i_1}$ by one and decreasing the
length of $N^{i_s}$ by one.  This has the effect of lifting the horizontal runs of $D$ to the west of $P$ 
by one unit.  

We claim that $F$ is contained in $\F(D')$ and that $\lambda(D') \subset \lambda(D)$.  We have
$\lambda(D') \subset \lambda(D)$  because $D'$ is constructed from $D$ by lifting a nonempty collection
of horizontal runs (namely, those between $N^{i_1}$ and $N^{i_s}$) up by one unit.  To see that 
$F$ is contained in $\F(D')$, we will observe that any laser involved in the realization of a diagonal of $F$ 
in $\F(D)$ corresponds to a laser realizing that same diagonal in $\F(D')$.  Indeed, let $R$ be a lattice
point of $D$ such that $d_D(R) \in F$.  If $R$ occurs to the northeast of $P$ in $D$, then $R$ is also a lattice
point in $D'$ and we have the equality of lasers $\ell_D(R) = \ell_{D'}(R)$ since $D$ and $D'$ agree
to the northeast of $P$.  It follows that $d_D(R) = d_{D'}(R) \in \F(D')$.  
On the other hand, if $R$ occurs to the southwest
of $P$ in $D$, then the vertical shift $R' = R + (0, 1)$ is a lattice point of $D'$ and the laser
$\ell_{D'}(R')$ is obtained by shifting the laser $\ell_D(R)$ up by one unit.  It follows that 
$d_D(R) = d_{D'}(R') \in \F(D')$.

The following diagram illustrates the construction in Case 1 when $(a, b) =  (5, 8)$.  
We have that 
$D = N^1 E^1 N^1 E^2 N^2 E^1 N^1 E^4$ and the lasers which give the face $F$ of
$\F(D)$ are shown.
The point $P$ is labeled and we have $s = 3$ and $i_s = 2$.  None of the lasers which originate
on lattice points to the west of $P$ hit $D$ to the east of $P$.  To form 
$D' = N^2 E^1 N^1 E^2 N^1 E^1 N^1 E^4$, we increase the length of the first vertical run
by one and decrease the length of the vertical run above $P$ by one.  We can vertically 
translate the lasers involved in the face $F$ to see that $F$ is also a face of $\F(D')$.
We have that 
$\lambda(D') = (4,3,1,0,0)$ and $\lambda(D) = (4,3,3,1,0)$, so that $\lambda(D') \subset \lambda(D)$.

\begin{center}
\includegraphics[scale = 0.4]{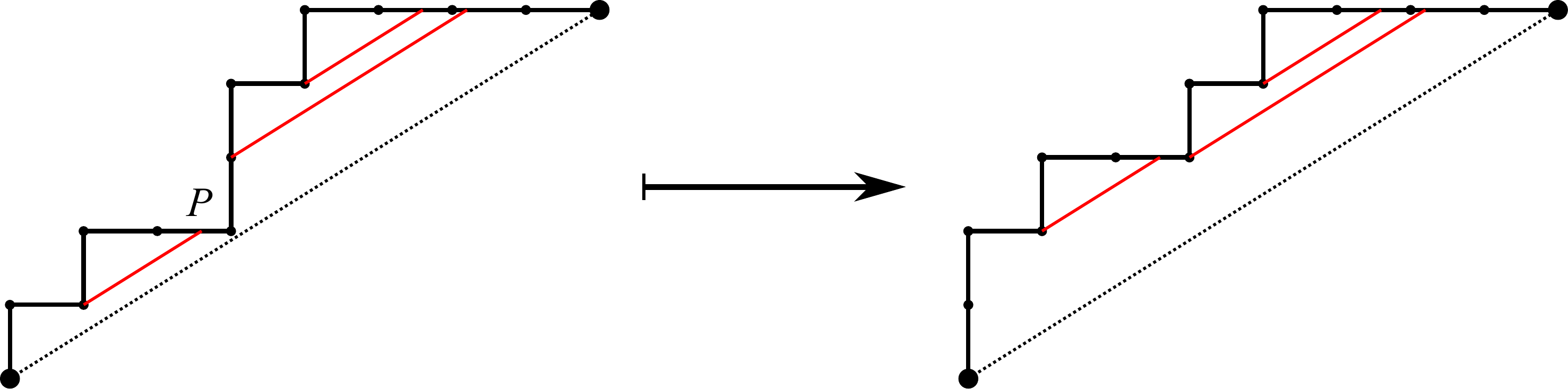}
\end{center}

{\bf Case 2:}  {\it There are some lattice points $Q$ on the vertical runs $N^{i_1}, \dots, N^{i_{s-1}}$ such that
the laser $\ell(Q)$ hits $D$ to the east of $P$.}

In this case, let $1 \leq m \leq s-1$ be maximal such that there exists a lattice point $P$ on $N^{i_m}$ such that
$\ell(P)$ hits $D$ to the east of $P$.  Form $D'$ by increasing the length of $N^{i_m}$ by one and decreasing 
the length of $N^{i_s}$ by one.  This has the effect of lifting the horizontal runs of $D$ between
$N^{i_m}$ and $P$ by one unit.  

We again claim that $F$ is contained in $\F(D')$ and that
$\lambda(D') \subset \lambda(D)$.  The argument is similar to Case 1, but slightly more complicated.
We have $\lambda(D') \subset \lambda(D)$ because $D'$ is formed from $D$ by moving the nonempty set 
of horizontal runs between $N^{i_m}$ and $N^{i_s}$ up one unit.
To see that $F$ is contained in $\F(D')$, let $R$ be a lattice point of $D$ such
that $d_D(R) \in F$.  We argue that $d_{D}(R) \in \F(D')$.    If $R$ occurs to the northeast of $P$, then $R$ is also
a lattice point of $D'$.  Since $D'$ agrees with $D$
to the northeast of $P$, we have the equality of lasers $\ell_D(R) = \ell_{D'}(R)$, so that 
$d_{D}(R) = d_{D'}(R) \in \F(D')$.
If $R$ lies to the southwest of $P$, our reasoning breaks up into further subcases.  If $R$ lies strictly east
of the vertical run $N^{i_m}$ (but southwest of $P$), then the shift 
$R' = R + (0, 1)$ is on the Dyck path $D'$.  By our choice of $m$, we know 
that the laser $\ell_D(R)$ hits $D$ to the east of $P$.  By our construction of $D'$, the laser $\ell_{D'}(R')$ is 
obtained by shifting the laser $\ell_D(R)$ up one unit, so that 
$d_D(R) = d_{D'}(R') \in \F(D')$.  
If $R$ lies strictly west of the vertical run $N^{i_m}$, then $R$ also lies on the lattice
path $D'$.  By our choice of $m$ and the fact that lasers from different lattice points do not cross, the laser
$\ell_D(R)$ hits $D$ either to the west of the vertical run $N^{i_m}$ or to the east of the valley $P$.  Since 
$D'$ is formed from $D$ by shifting $D$ up between these points, we get that
$\ell_D(R) = \ell_{D'}(R)$ in either case, so that $d_D(R) = d_{D'}(R') \in \F(D')$.  
Finally, suppose that $R$ lies on the 
vertical run $N^{i_m}$.  Depending on where on this run $R$ lies, the laser $\ell_D(R)$ may hit $D$
either to the west of $P$ or to the east of $P$.  Suppose $\ell_D(R)$ hits $D$ to the west of $P$.  
The construction of $D'$ implies that the shift $R' = R + (0, 1)$ is a lattice point on $D'$ and
that the laser $\ell_{D'}(R')$ is obtained by shifting the laser $\ell_D(R)$ up by one unit.  
It follows that $d_D(R) = d_{D'}(R') \in \F(D')$. 
Suppose that $\ell_D(R)$ hits $D$ to the east of $P$.  By the construction of $D'$ we know that 
$R$ is also a lattice point of $D'$ and that we have the equality of lasers $\ell_D(R) = \ell_{D'}(R')$.  
This means that $d_D(R) = d_D(R') \in \F(D')$.

The following diagram illustrates the construction in Case 2 when $(a, b) = (5, 8)$.  We have that 
$D = N^1 E^1 N^1 E^1 N^1 E^2 N^2 E^4$ and the lasers which give the face $F$ of $\F(D)$ are shown.
The point $P$ is labeled and we have $s = 4$ and $m = 2$.  To form $D' = N^1 E^1 N^2 E^1 N^1 E^2 N^1 E^4$,
we increase the length of the $m^{th} = 2^{nd}$ vertical run by one and decrease the length of the vertical
run above $P$ by one.  As in Case 1, we can vertically translate the lasers involved in the face $F$ to
see that $F$ is also a face of $\F(D')$.  We have that
$\lambda(D) = (4,4,2,1,0)$ and $\lambda(D') = (4,2,1,1,0)$, so
that $\lambda(D') \subset \lambda(D)$.

\begin{center}
\includegraphics[scale = 0.4]{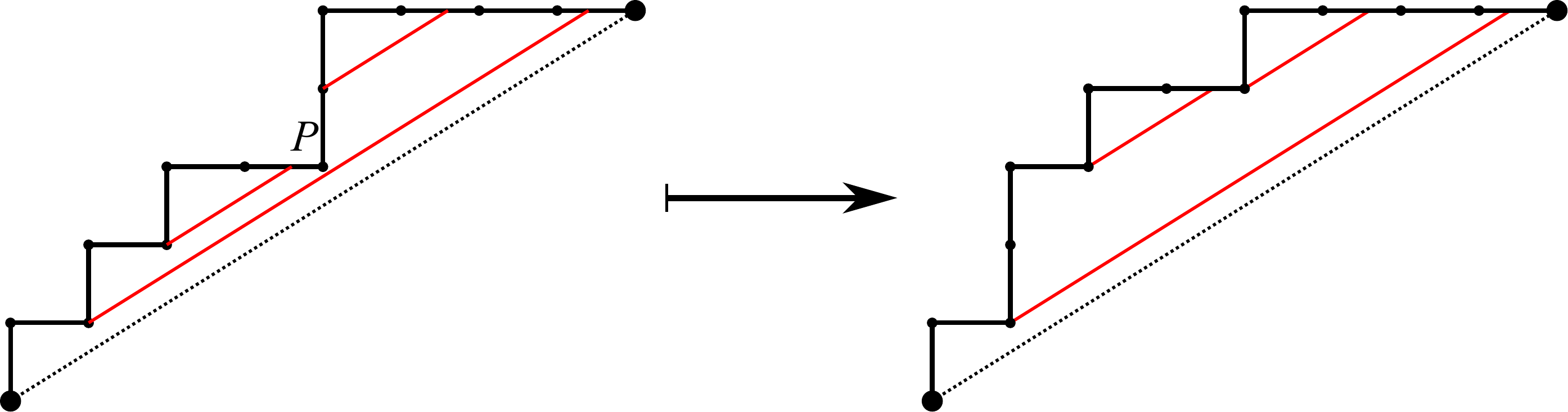}
\end{center}

\end{proof}

The path $\D(F)$ in Lemma~\ref{valley-lemma} will be called the \textbf{\textit{valley path}} of $F$. 
The path $\D(F)$ can be characterized as the unique Dyck path among those Dyck paths $D$
such that $F \subseteq \F(D)$ which minimizes $\lambda(D)$ in Young's Lattice, but we will not 
need this characterization. 
By convention, we will set the valley path of the empty face to be the path $N^a E^b$ consisting of 
$a$ north steps followed by $b$ east steps.

We return to the more general situation of a face $F \in \Asshat(a,b)$.  The following algorithm 
constructs the valley path $\D(F)$ if $F$ is a face of $\Ass(a,b)$, or returns the statement that 
$F$ is not a face of $\Ass(a,b)$.  Roughly speaking, we attempt to construct $\D(F)$ ``backwards" by 
starting with the empty path at $(b, a)$ and working our way towards $(0, 0)$.  To emphasize the 
backwards nature of this construction, we refer to the steps added to $\D(F)$ as south and west steps
rather than north and east steps.  Starting at $(b, a)$, we go west until we hit an $x$-coordinate $i$ such that
$ij$ is a diagonal in $F$ for some $j$.  We then go south until we pick up all of the diagonals in $F$ of the form
$ij$, at which point we go west again until we hit another $x$-coordinate $i'$ such that 
$i'j'$ is a diagonal in $F$ for some $j'$, etc.  If it ever happens that a diagonal in $F$ cannot be achieved by 
going south, we go south until we cross the line $y = \frac{a}{b}x$ and return the statement that $F$
is not a face of $\Ass(a,b)$.  Otherwise, we will eventually reach the origin $(0, 0)$ and recover
the path $\D(F)$ (and the fact that $F$ is a face of $\Ass(a,b)$).

PATH BUILDING ALGORITHM: \\ \indent
INPUT: A face $F$ of $\Asshat(a,b)$. \\ \indent
OUTPUT:  The valley path $\D(F)$ or the statement that 
$F$ is not a face of $\Ass(a,b)$. 
\begin{enumerate}
\item Initialize $\D(F)$ to be the empty lattice path starting and ending at $(b,a)$ and initialize $i = b$.
While $i \geq 0$, do the following.
\begin{enumerate}
\item If $i > 0$ and
there are no diagonals in $F$ of the form $i  j$, add a west step to the southwest end of $\D(F)$, so that
$\D(F)$ has southwest corner with $x$-coordinate $i-1$.
\item If $i = 0$ and there are no diagonals in $F$ of the form $0  j$, add south steps to the southwest end of 
$\D(F)$ until $\D(F)$ reaches $(0, 0)$.
\item If there are diagonals in $F$ of the form $i  j$, express these diagonals as 
$i  j_1, \dots, i  j_r$, where $j_1 < \dots < j_r$.  Add south steps to the southwest corner of $\D(F)$ until 
there exist lattice points $P$ on $\D(F)$ whose laser diagonals correspond to all of the diagonals
$i  j_1, \dots, i  j_r$ or until $\D(F)$ goes below the line $y = \frac{a}{b}x$. If $i > 0$, add 
a single west step to the southwest corner of $\D(F)$.  If $i = 0$, add south steps to the southwest
corner of $\D(F)$ until $\D(F)$ is weakly below $(0, 0)$.
\item  If $\D(F)$ lies weakly above the line $y = \frac{a}{b}x$, replace $i$ by $i-1$.  Otherwise, terminate the 
algorithm and return ``$F$ is not a face of $\Ass(a,b)$".
\end{enumerate}
\item  Return ``$F$ is contained in the facet $\F(\D(F))$ of $\Ass(a,b)$ and $\D(F)$ is the valley path of $F$".
\end{enumerate}

A couple comments on this algorithm are in order.  At the start of every iteration of the loop (1), the path $\D(F)$
starts at $(b, a)$ and ends at a point with $x$-coordinate $i$ while staying above the line $y = \frac{a}{b} x$.  
In step (1c), we consider firing lasers 
of slope $\frac{a}{b}$
to the northeast from north steps with $x$-coordinate $i$.  Since $\D(F)$ has northeast corner $(b, a)$ 
and stays above $y = \frac{a}{b} x$, the laser diagonals referenced in this step are well defined.

\begin{example}
We examine a case where the path building algorithm
does not terminate in a Dyck path.  Let $(a, b) = (5, 8)$ and consider the face
$F = \{57, 24, 05, 04\}$ of $\Asshat(5,8)$.  We initialize the lattice path $\D(F)$ to be the empty path 
at the point $(8, 5)$.  As the algorithm progresses, the lattice path $\D(F)$ changes as shown. 
For ease of reading, $x$-coordinates and laser diagonals involved in $F$ are given.  

\begin{center}
\includegraphics[scale = 0.4]{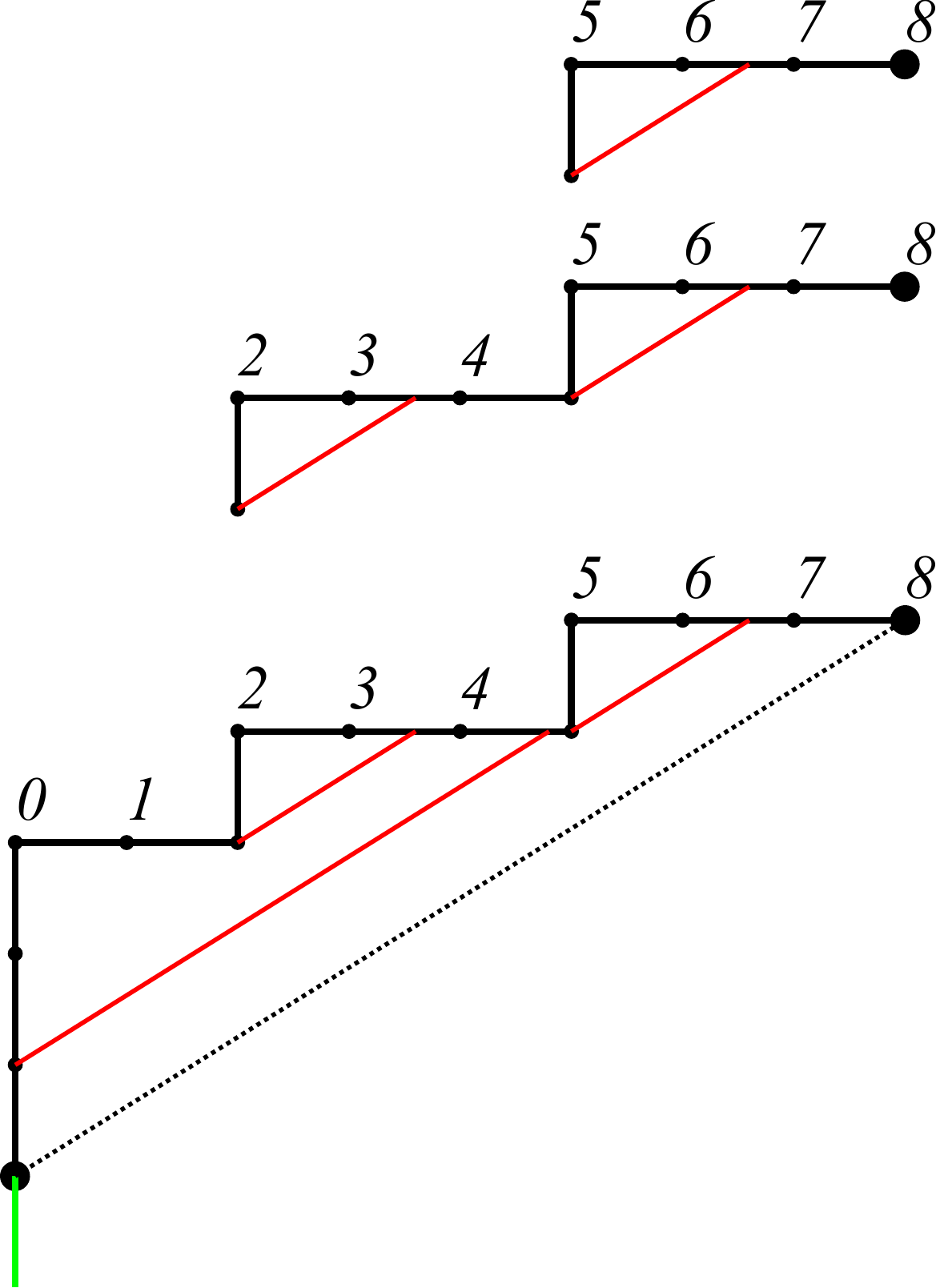}
\end{center}

We start at the lattice point $(8, 5)$ and build the path $\D(F)$ from northeast to southwest, working
west until hitting an $x$-coordinate corresponding to the smaller endpoint of at least one diagonal in
$F$ and then working south until we get every diagonal of $F$ with that smaller endpoint appearing 
as a laser (or, if this is impossible, working south until we go below the line $y = \frac{5}{8} x$).

In particular, starting at $(8, 5)$, we go west until we hit the line $x = 5$.  Then we go south one unit
to pick up the diagonal $58 \in F$.  Then we go west until we hit the line $x = 2$.  We go south one
unit to pick up the diagonal $24 \in F$.  Then we go west until we hit the line $x = 0$.  While we can
pick up the diagonal $05 \in F$ by going south, it is impossible to pick up the diagonal $04 \in F$.  As a result
of this `failure', we continue going south until we cross the line $y = \frac{5}{8} x$ (this crossing 
step is shown in green) and conclude that $F \notin \Ass(5,8)$.
\end{example}

\begin{lemma}
\label{valid}
The path building algorithm is correct. 
\end{lemma}

\begin{proof}
When $F \in \Ass(a,b)$, the path building algorithm constructs the unique lattice path $\D(F)$ satisfying
the conditions of Lemma~\ref{valley-lemma}.  Therefore the algorithm never returns the statement
that $F \in \Asshat(a,b)$ is not contained in $\Ass(a,b)$ when in fact $F \in \Ass(a,b)$.  On the other hand,
if $F \in \Asshat(a,b)$ and the path building algorithm terminates with a valid
$a,b$-Dyck path $\D(F)$, it is obvious that $F$ is a face of $\F(\D(F))$ so that
$F \in \Ass(a,b)$.
\end{proof}

The path building algorithm gives a constructive way to decide whether a face $F \in \Asshat(a,b)$
is contained in the subcomplex $\Ass(a,b)$. It can also be used to prove that 
$\Ass(a,b)$ is flag.

\begin{proof} (of Proposition~\ref{flag})
Let $F$ be a collection of $a,b$-admissible diagonals which is not a face of $\Ass(a,b)$.  We need to show that
there exist two diagonals $d, d' \in F$ such that $\{d, d'\}$ is not a face of $\Ass(a,b)$.  If there exist two
crossing diagonals $d, d' \in F$ this is clear, so we may assume that $F$ is a face of $\Asshat(a,b)$.

We apply the path building algorithm to $F$.  Since $F$ is not a face of $\Ass(a,b)$, the loop in step (1) of this 
algorithm must break for some value of $i$ such that $0 \leq i \leq b$.  Since for any input $F$
the path building algorithm starts by adding $\lfloor \frac{b}{a} \rfloor + 1$ west steps
to $\D(F)$, we must have that
$0 \leq i \leq b - \lfloor \frac{b}{a} \rfloor - 1$.  We call a diagonal $jk$ in $F$ \textbf{\textit{redundant}}
if any of the following conditions hold:
\begin{itemize}
\item $j < i$,
\item $j > i$ and there exists a diagonal $jk'$ in $F$ with $k' > k$, or
\item $j > i$ and there exists a diagonal $j'k'$ in $F$ with $i < j' < j$ and $k' \geq k$.
\end{itemize}
We define a subset $\overline{F}$ of $F$ by
\begin{equation}
\overline{F} := \{ d \in F \,:\, \text{$d$ is not redundant} \}.
\end{equation}
We define a further subset $\overline{\overline{F}}$ of $\overline{F}$ by
\begin{equation}
\overline{\overline{F}} := \{ d \in \overline{F} \,:\, \text{$d$ is of the form $jk$ for $j > i$} \}.
\end{equation}
It is easy to see that the loop in the path building algorithm breaks at the same value $i$ for $\overline{F}$
as it did for $F$. When this loop breaks, the  path $\D(\overline{F})$ is a 
lattice path starting at $(b,a)$ and ending at the
northernmost
lattice point of the line $x = i$ which is strictly below $y = \frac{a}{b} x$.  Reading $\D(\overline{F})$ `from left to right',
we can factor the portion of $\D(\overline{F})$ starting 
with an east step at $x = i$ into nonempty horizontal and vertical runs
as $E^{i_1} N^{i_1} \dots E^{i_r} N^{i_r} E^{i_{r+1}}$.  By the 
definition of a redundant diagonal, there are lasers fired off of this portion of $\D(\overline{F})$ only
at the valleys $P_s$ between $E^{i_s}$ and $N^{i_s}$ for $s = 1, 2, \dots, r$.  Moreover, the path building
algorithm shows that $\overline{\overline{F}}$ is contained in $\Ass(a,b)$ and that
$\D(\overline{\overline{F}})$ has the form $N^k E^i E^{i_1} N^{i_1} \dots E^{i_r} N^{i_r} E^{i_{r+1}}$.

Let $Q$ denote the westernmost
point of the horizontal step sequence $E^{i_1}$ (so that $Q$ is a lattice point on either of
$\D(\overline{F})$ or $\D(\overline{\overline{F}})$).  
Since $b > a$, by slope considerations there are a total of  
$\lfloor \frac{a(b-i)}{b} \rfloor - 1$ 
distinct
$a,b$-admissible diagonals of the form $i  j$.    
Since $Q$ is below the vertical runs $N^{i_1}, \dots, N^{i_r}$ of lengths $i_1, \dots, i_r$, the correctness
of the path building algorithm implies that
there are exactly $\lfloor \frac{a(b-i)}{b} \rfloor - 1 -  (i_1 + \dots + i_r)$ values of $j$ such that
$i  j$ is admissible and $\overline{\overline{F}} \cup \{i j\}$ is contained in $\Ass(a,b)$.  
(These values of $j$ correspond to the lattice points on the interior of the vertical segment between
$Q$ and the line $y = \frac{a}{b}x$.)

On the other hand, by 
coprimality, for $s = 1, 2, \dots, r$ the laser emanating from the valley $P_s$ gives exactly $i_s$
$a,b$-admissible diagonals of the form $i  j$ such that $\overline{\overline{F}} \cup \{i  j\}$ is not contained
in $\Ass(a,b)$.  
To see this, consider the vertical line segment $L_s$ of length $i_s$ underlying the vertical run $N^{i_s}$
above the valley $P_s$. 
Let $L'_s$ be the vertical line segment on $x = i$ obtained by
translating the points of $L_s$ southwest along a line of slope $\frac{a}{b}$.
By coprimality, the interior of the line 
segment $L'_s$ contains exactly $i_s$ lattice points $P$ on the line $x = i$.  
The diagonals $d_{\D(\overline{\overline{F}})}(P)$ of these points (constructed with respect to the lattice path
$\D(\overline{\overline{F}})$)
cannot be added to $\overline{\overline{F}}$ to get a face in $\Ass(a,b)$.  Graphically, the laser
$\ell(P_s)$ `cuts off' these diagonals.
By the definition of redundant diagonals, if $s \neq s'$ are the indices of two different valleys, the two sets of 
diagonals cut off by $\ell(P_s)$ and $\ell(P_{s'})$ are disjoint.  This means that the set of all lasers
$\ell(P_s)$, as $P_s$ varies over the valleys of $\D(\overline{F})$, cut off $i_1 + \cdots + i_r$ diagonals
of the form $ij$.

 By the pigeonhole principle and the reasoning of the last two paragraphs,
 if $\overline{\overline{F}} \cup \{ij \}$ is in the complement 
$\Asshat(a,b) - \Ass(a,b)$ for some diagonal $i  j$ in $F$, there exists $1 \leq s \leq r$ such that the laser
$\ell(P_s)$ cuts off the diagonal $ij$.  It is straightforward
 to check using the 
 path building algorithm that the $2$-element subset $\{ i  j, d(P_s) \}$ of $F$ is not contained in $\Ass(a,b)$,
 so $F$ is not an empty face.
 \end{proof}
 
 \begin{example}
 As an example of the argument in the proof of Proposition~\ref{flag}, we let
 $(a, b) = (5,8)$ and $F = \{57, 48, 24, 04\}$.  Since the diagonals in $F$ do not cross,
 we have that $F$ is a face of $\Asshat(5,8)$.  To check whether $F$ is a face of $\Ass(5,8)$, we apply
 the path building algorithm to $F$.  The loop in step 1 of the path building algorithm applied
 to $F$ breaks at the $x$-coordinate $x = 0$, at which point $\D(F)$ is the following lattice path.
 We conclude that $F$ is not contained in $\Ass(5,8)$.
 
 \begin{center}
 \includegraphics[scale = 0.5]{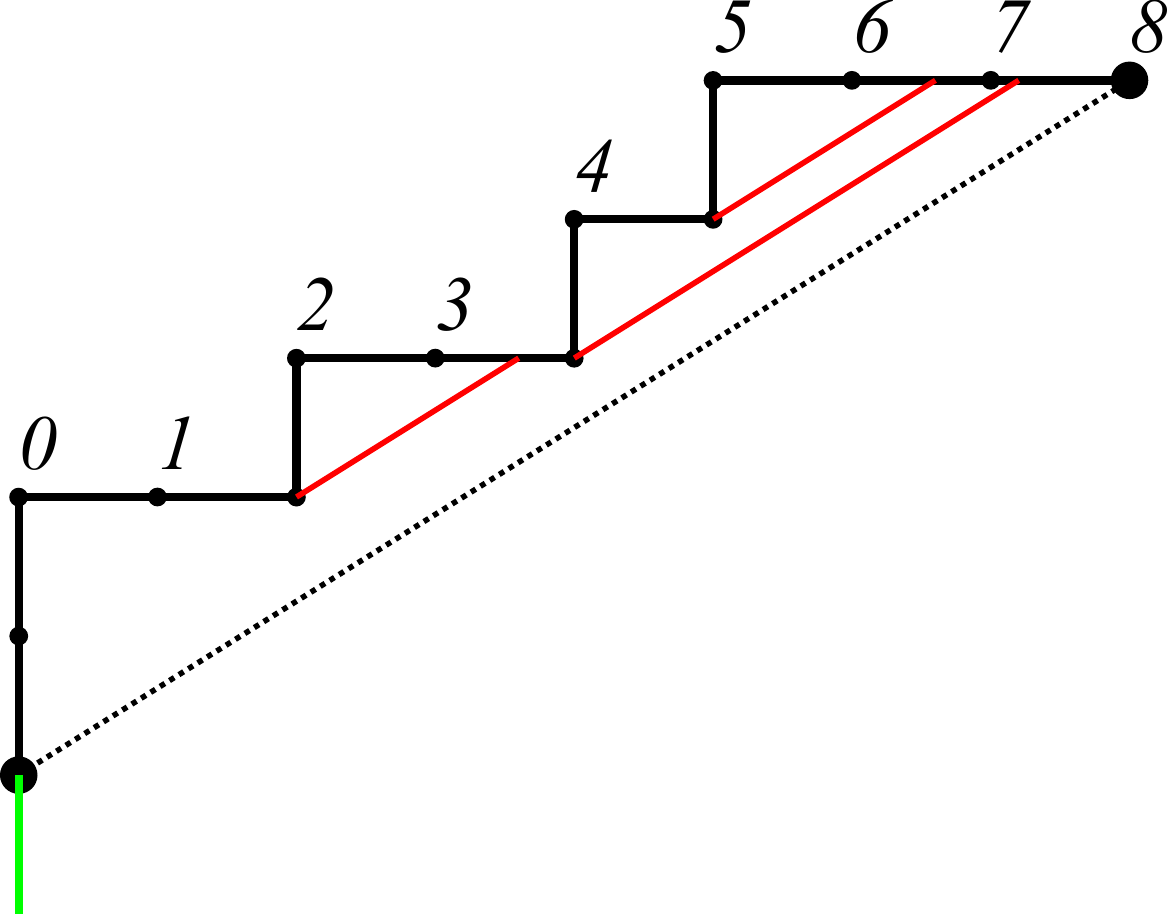}
 \end{center}
 
 As before, the green lattice step which goes below the line $y= \frac{5}{8}x$ indicates that the 
 path building algorithm
 breaks on the line $x = 0$.  To form $\overline{F}$ from $F$, we remove all redundant diagonals from $F$.
 The only redundant diagonal is $57$, so we have 
 $\overline{F} = \{48, 24, 04\}$.  In terms of lattice paths, the redundancy of $57$ corresponds to the fact
 that the laser corresponding to $57$ in the above lattice path is entirely to the northwest of the laser
 corresponding to $48$.  Applying the path building algorithm to $\overline{F}$ again
 leads to the loop in step $1$ breaking at $x = 0$, at which point $\D(\overline{F})$ is the following 
 lattice path.  Observe that the only lasers fired from $\D(\overline{F})$ which give rise to diagonals
 in $\overline{F}$ are fired from valleys.
 
 \begin{center}
 \includegraphics[scale = 0.5]{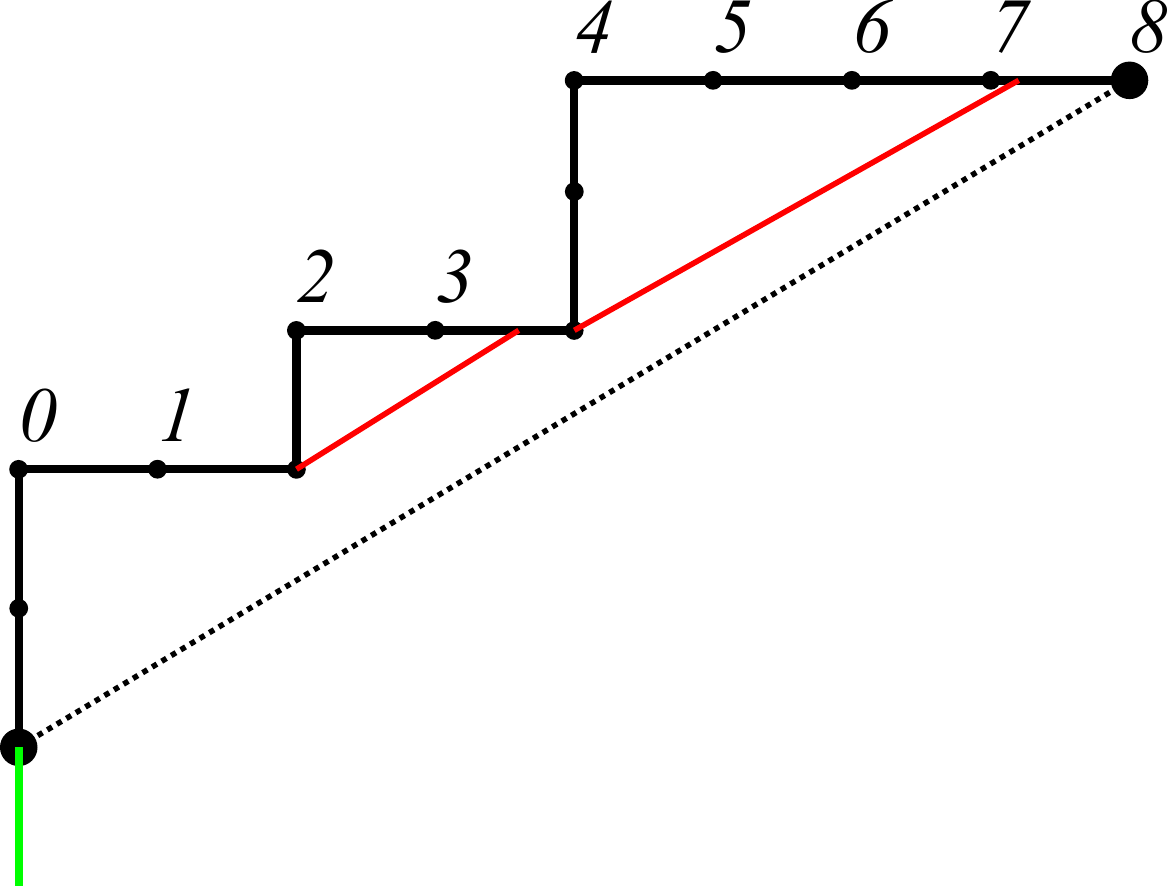}
 \end{center}
 
 To form $\overline{\overline{F}}$ from $\overline{F}$, we remove all diagonals in $\overline{F}$ with 
 smaller vertex $0$.  Therefore, we have that $\overline{\overline{F}} = \{48, 24\}$.  We have that 
 $\overline{\overline{F}}$ is a face of $\Ass(5,8)$ and $\D(\overline{\overline{F}})$ is the above lattice
 path without the green step.

 We claim that there exist diagonals $d \in \overline{\overline{F}}$ and $d' \in \overline{F}$ such that
 the smaller boundary point of $d'$ is $0$ and $\{d, d'\}$ is not a face of $\Ass(5,8)$ 
 Since $d, d' \in F$, this would imply
  that
 $F$ is not an empty face.  To see this, let $P_1 = (2,2)$ and $P_2 = (4,3)$ be the valleys of the above
 lattice path $\D(\overline{F})$, read from right to left.  By the definition of a redundant diagonal,
 the only lasers fired from $\D(\overline{F})$ are fired from the valleys of $\D(\overline{F})$ and the 
 laser $\ell_{\D(\overline{F})}(P)$ fired from a valley $P$ hits $\D(\overline{F})$ on the first horizontal
 run to the northeast of $P$.   
 
 We consider the vertical line segments $L_1$ and $L_2$ given by the vertical runs of
 $\D(\overline{F})$ which lie above the valleys $P_1$ and $P_2$.  In our example, the segment $L_1$ has length
 $1$ and the segment $L_2$ has length $2$.  We translate these segments along the line $y = \frac{5}{8}x$
 until they lie on the line $x = 0$, obtaining the line segments $L_1'$ and $L_2'$.  This is shown in the diagram
 below (where the line $x = 0$ is drawn in green).  
 This diagram also shows that the various $L_s'$ line segments need not be disjoint - indeed, 
 $L_1'$ is contained in $L_2'$.  By coprimality, none of the $L_s'$ line segments has a lattice point 
 on its boundary.
 
 \begin{center}
 \includegraphics[scale = 0.5]{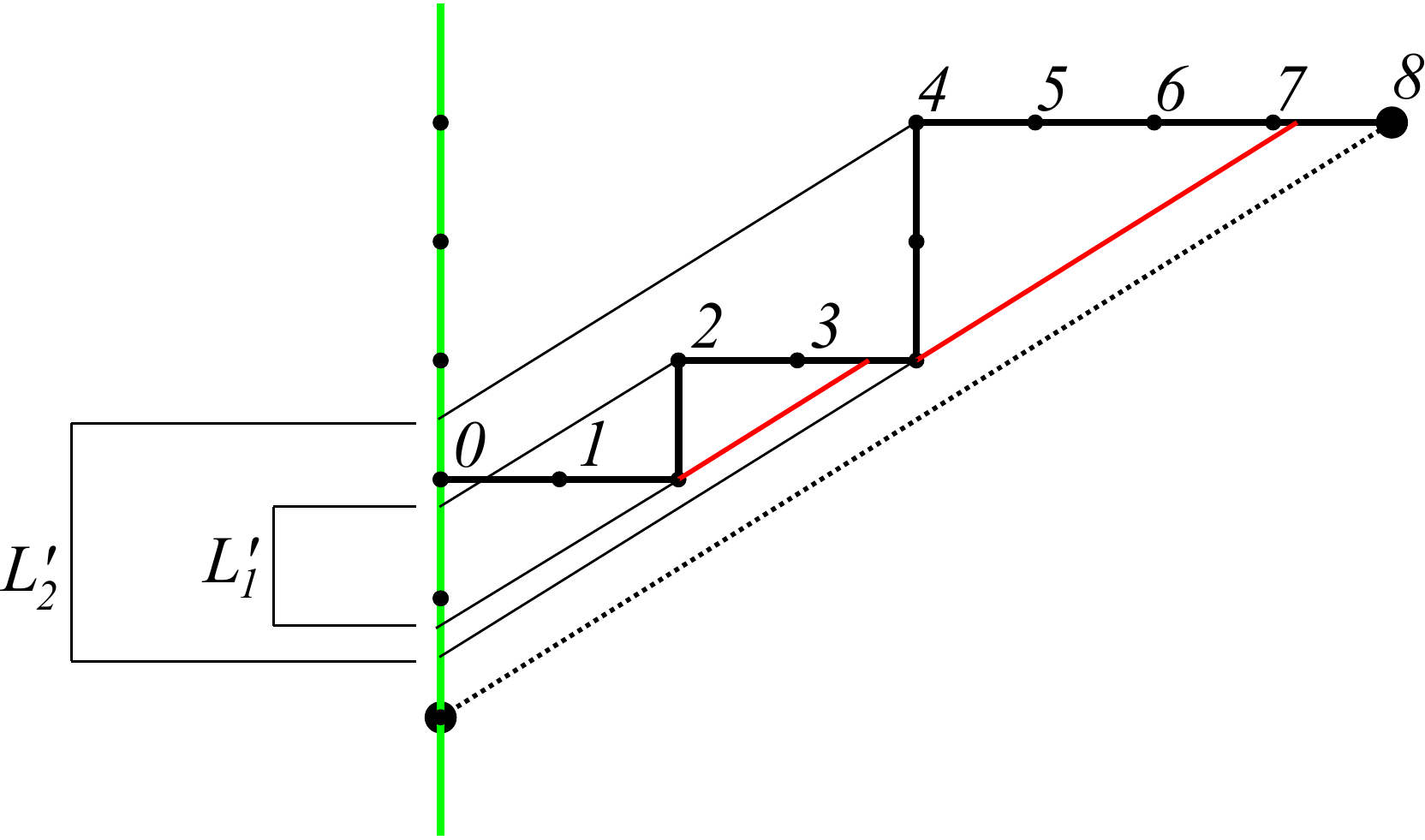}
 \end{center}
 
For any valley $P_s$, the laser $\ell(P_s)$ cuts off a single diagonal of the form $0j$ for each lattice
point in the interior of $L_s'$, where $j$ is an $x$-coordinate of the first horizontal run to the northeast of $P_s$.
Looking at the above diagram, we see that the laser $\ell(P_1)$ cuts of the diagonal $04$ (corresponding
to the lattice point $(0,1)$ in the interior of $L_1'$)
and $\ell(P_2)$ cuts off the diagonals $05$ and $07$ (corresponding to the lattice points
$(0,2)$ and $(0,1)$ in the interior of $L_2'$, respectively).

While the line segments $L_s'$ may intersect for various 
values of $s$, the lack of redundant edges in $\overline{F}$ forces the sets of diagonals cut off by the lasers
$\ell(P_s)$ to be disjoint.  This means that the lasers for all of the valleys $P_s$ cut off a total of $2 + 1 = 3$ 
diagonals.  On the other hand, there are $\lfloor \frac{5(8-0)}{8} \rfloor - 1 = 4$ total $5,8$-admissible diagonals
of the form $0j$ and only $\lfloor \frac{5(8-0)}{8} \rfloor - 1 - (2+1) = 1$ of these diagonals 
(namely, the diagonal $02$) is not cut off by one of the lasers.  This single diagonal $02$ 
can be realized as a laser
diagonal on $\D(\overline{F})$ by firing a laser from the $\lfloor \frac{5(8-0)}{8} \rfloor - 1 - (2+1) = 1$ lattice
point (namely, the lattice point $(0,1)$) on the line $x = 0$ which lies on $\D(\overline{F})$.  Since the path building
algorithm breaks on $\overline{F}$ on the line $x = 0$, this means that $\overline{F}$ must contain a diagonal
$0j$ cut off by one of the lasers $\ell(P_s)$.  In fact, the face $\overline{F}$ contains the diagonal $04$, which is cut
off by the laser $\ell(P_1)$.  We can use the path building algorithm to see that 
$\{ d(P_1), 04 \} = \{24, 04\}$ is not a face of $\Ass(5,8)$, so that $F$ is not an empty face.
\end{example}

\section{Obstruction graphs}
\label{Obstruction graphs}

Proposition~\ref{flag} tells us that the obstructions preventing a face $F$ of $\Asshat(a,b)$
from being a face of $\Ass(a,b)$ are local in nature.  To keep track of these obstructions, we make use of a graph.

\begin{defn}
Let $a < b$ be coprime positive integers.  The \textbf{\textit{obstruction graph}} $\OG(a,b)$ has vertex set 
given by the collection of $a,b$-admissible diagonals in $\PP_{b+1}$ and an edge 
connecting two distinct $a,b$-admissible diagonals $d$ and $d'$ if $\{d, d'\}$ is contained in the complement
$\Asshat(a,b) - \Ass(a,b)$.
\end{defn}

By Proposition~\ref{flag}, a face $F \in \Asshat(a,b)$ is contained in $\Ass(a,b)$ if and only if 
$F$ does not contain any edges of $\OG(a,b)$.  
In the language of deletions, we have that
$\dl_{\Asshat(a,b)}(E(\OG(a,b))) = \Ass(a,b)$, where
$E(\OG(a,b))$ denotes the set of edges of the obstruction graph $\OG(a,b)$.

\begin{example}
When $b \equiv 1$ modulo $a$, the graph $\OG(a,b)$ has no edges and $\Asshat(a,b) = \Ass(a,b)$.
The obstruction graph $\OG(3,5)$ has as its vertex set the collection of $3,5$-admissible diagonals
of $\PP_6$ and edges $\{04 - 24, 15 - 35\}$.
Figure~\ref{obstruction} shows the obstruction graph $\OG(5,8)$.  The rectangles in this diagram 
correspond to a decomposition of $\OG(a,b)$ given in Lemma~\ref{same-larger-endpoint}.
\end{example}

\begin{figure}
\centering
\includegraphics[scale = 0.6]{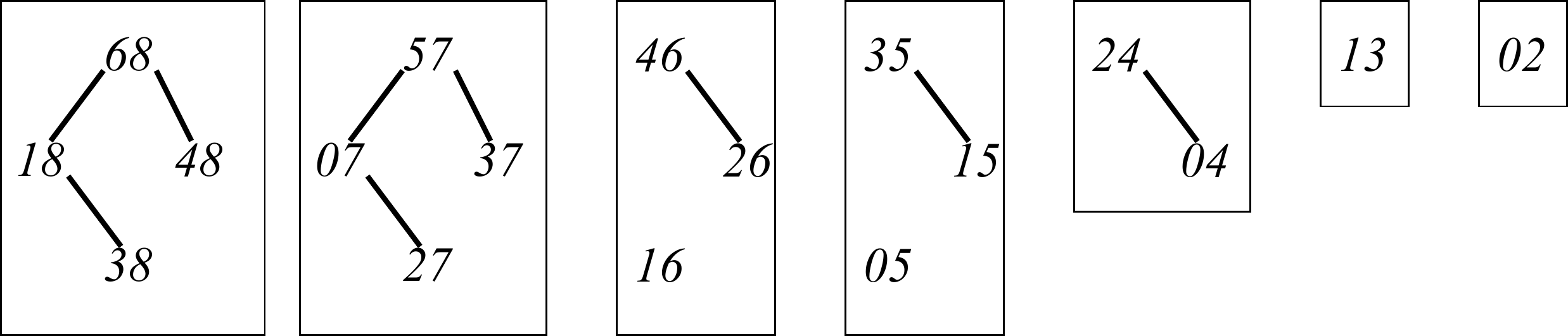}
\caption{The obstruction graph $\OG(5,8)$.}
\label{obstruction}
\end{figure}

Our strategy for proving that $\Asshat(a,b)$ collapses
onto $\Ass(a,b)$ is to give a total order 
$E(\OG(a,b)) = \{e_1 \prec \cdots \prec e_N\}$ on the edge set of the obstruction graph.  For $0 \leq r \leq N$, we
define complexes $\Ass(a,b)_r$ inductively by 
$\Ass(a,b)_N := \Asshat(a,b)$ and
$\Ass(a,b)_r := \dl_{\Ass(a,b)_{r+1}}(e_{r+1})$ for $0 \leq r \leq N-1$.  It follows that
we have a chain of subcomplexes
$\Ass(a,b)_0 \subset \cdots \subset \Ass(a,b)_N$ with 
$\Ass(a,b)_0 = \Ass(a,b)$ and $\Ass(a,b)_N = \Asshat(a,b)$.  We will prove that
$\Ass(a,b)_r$ collapses onto $\Ass(a,b)_{r+1}$ for all $1 \leq r \leq N$.  Before we define this 
order and to help us realize these collapsings, we will prove several results on the structure of
obstruction graphs.

A glance at Figure~\ref{obstruction} shows that if $\{ij, km\}$ is an edge of $\OG(5,8)$, then $j = m$.
This is not a coincidence.  Our main structural result about obstruction graphs is as follows.

\begin{lemma}
\label{same-larger-endpoint}
Let $i j$ and $km$ be two $a,b$-admissible diagonals which form an edge of 
$\OG(a,b)$.  We have that $j = m$.
\end{lemma}

\begin{proof}
If $i = k$, we claim that the set $\{ ij, im\}$ is a face of $\Ass(a,b)$.  Let
$D$ be the $a,b$-Dyck path shown below which has a vertical run of maximal length at $x = i$.

\begin{center}
\includegraphics[scale = 0.5]{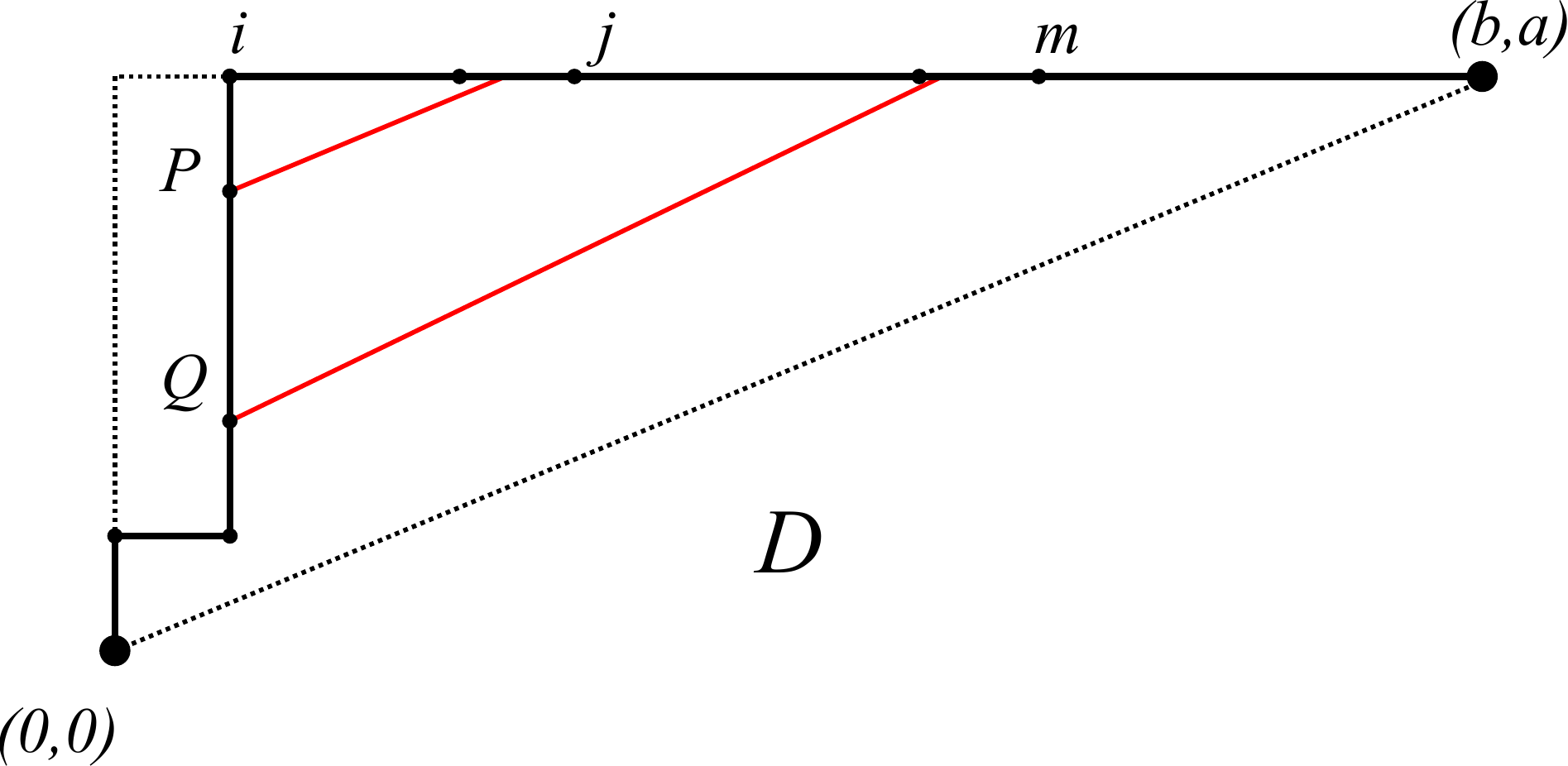}
\end{center}

Since $ij$ and $im$ are $a,b$-admissible, there must be points $P$ and $Q$ on $D$
with $x$-coordinate $i$ whose laser diagonals are $d(P) = ij$ and $d(Q) = im$.  We conclude
that $\{ij, im\} \subseteq \F(D)$ and $\{ij, im\}$ is a face of $\Ass(a,b)$.
This is a contradiction.  

We conclude that $i \neq k$.
Without loss of generality, we may assume that $i < k$.  Since the diagonals $i  j$ and $k m$ do not 
cross, we have that either $j \leq k$ or
$m \leq j$.  
We argue first that $j \leq k$ is impossible.

Suppose that $i < k$ and $j \leq k$.  We claim that $\{ij, km\}$ is a face of $\Ass(a,b)$, which
is a contradiction.  Since $km$ is admissible, there exists an $a,b$-Dyck path
$D_1$ of the form $D_1 = N^r E^k N^{a-r} E^{b-k}$ (for some $r > 0$) such that the laser fired from the
unique valley of $D_1$ gives rise to the diagonal $km$.  Similarly, there exists another
$a,b$-Dyck path $D_2$ of the form $D_2 = N^s E^i N^{a-s} E^{b-i}$ such that
the laser fired just after the $E^i$ horizontal run (which is empty if $i = 0$) gives rise to the 
diagonal $ij$.  Let $D_3$ be the lattice path 
$D_3 = N^{r+s-a} E^i N^{a-s} E^{k-i} N^{a-r} E^{b-k}$.  Geometrically, the path $D_3$ is obtained from
the path $D_1$ by lowering the easternmost $i$ steps of the horizontal run $E^k$ by
$a-s$ units.  This procedure gives rise to a valid $a,b$-Dyck path $D_3$ (i.e., the lattice path 
$D_3$ stays above the line $y = \frac{a}{b} x$) because the laser fired from the point just after the horizontal
run $E^i$ on $D_3$ must hit $D_3$ in the step on 
the horizontal run
$E^{k-i}$ whose right endpoint has 
$x$-coordinate $j$ (here we use $j \leq k$ and the construction of $D_3$). 
Since $D_1$ is an $a,b$-Dyck path, we conclude that this laser is above the line
$y = \frac{a}{b} x$ and $D_3$ is also an $a,b$-Dyck path.  But now $ij$ and $km$ are both laser diagonals
arising from $D_3$, so that $\{ij, km\} \subseteq \F(D_3)$, which is a contradiction.

By the above two paragraphs, we have that $i < k$ and $m \leq j$.
We will argue that $j \leq m$ as well.

Let $P'$ be the unique lattice point with $x$-coordinate $i$ such that the laser $\ell(P')$ of slope
$\frac{a}{b}$ hits the 
horizontal line $y = a$ at the horizontal step whose right endpoint has $x$-coordinate $j$.
Similarly, let $Q'$ be the unique lattice point with $x$-coordinate $j$ such that the laser
$\ell(Q')$ hits the horizontal line $y = a$ at the horizontal step whose right endpoint has 
$x$-coordinate $m$.  If the $y$-coordinate of $P'$ is greater than or equal to the $y$-coordinate 
of $Q'$, then clearly $j \leq m$.  
So we may assume that the $y$-coordinate of $P'$ is strictly less than the $y$-coordinate of $Q'$.

Let $D'$ be the $a,b$-Dyck path shown below with lasers fired from and valleys at $P'$ and $Q'$.  
We extend the laser $\ell(P')$ in a dashed fashion to the line $y = b$.

\begin{center}
\includegraphics[scale = 0.5]{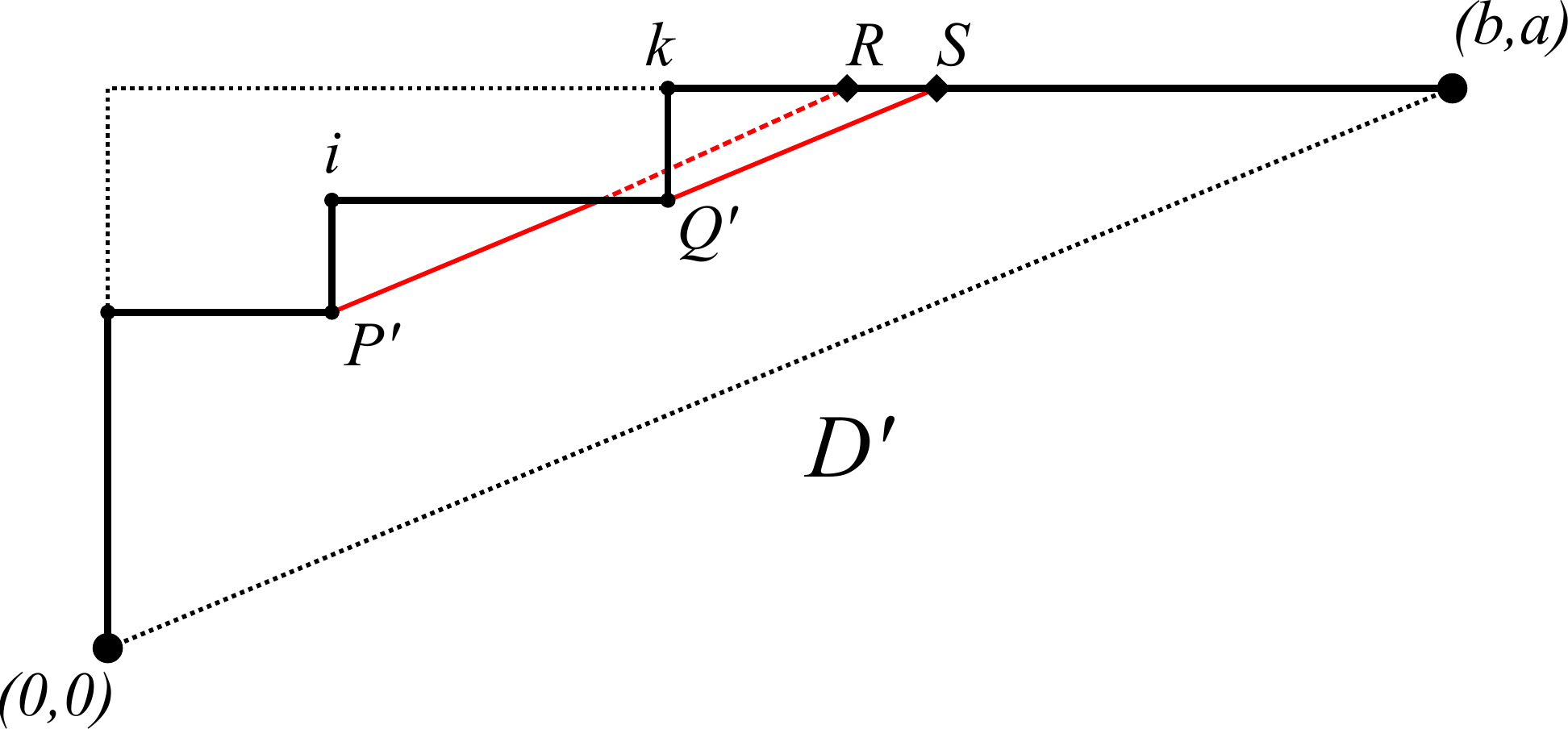}
\end{center}

We know that the right endpoint of the step containing $R$ has $x$-coordinate $j$ and that the right
coordinate of the step containing $S$ has $x$-coordinate $m$.  Since $\{ij, km\}$ is not a face of 
$\Ass(a,b)$, the laser $\ell(P')$ must hit the horizontal run of $D'$ which lies to the right of $Q'$.
This means that the laser $\ell(P')$ (and its extension) must lie to the left of $\ell(Q')$.  This implies that
$j \leq m$.
\end{proof}

Lemma~\ref{same-larger-endpoint} implies that the obstruction graph $\OG(a,b)$ breaks up into  
components $\OG(a,b)_0 \uplus \OG(a,b)_1 \uplus \dots \uplus \OG(a,b)_b$, where
$\OG(a,b)_m$ has edges consisting of two diagonals whose common larger endpoint is $m$.  
Figure~\ref{obstruction} gives these components as rectangles.  Observe that these components
may be disconnected (for example, the component $\OG(5,8)_6$ is disconnected) and
may be empty (for example, the components $\OG(5,8)_0$ and $\OG(5,8)_1$ are empty).  Although
we will not need this result, we remark that the obstruction graph $\OG(a,b)$ is determined
by the component $\OG(a,b)_b$.  If $d = ij$ is a diagonal in $\PP_{b+1}$ and $k \geq 0$ is such that
$i-k \geq 0$, we let $d - k$ denote the diagonal $(i-k)(j-k)$ in $\PP_{b+1}$.  If $i-k < 0$, we leave
$d - k$ undefined.  We define $d + k$ similarly.

\begin{lemma}
\label{determined-by-largest-component}
For $a < b$ coprime and $0 \leq m \leq b$, the component $\OG(a,b)_m$ of the obstruction graph
$\OG(a,b)$ has vertex set given by the diagonals 
\begin{equation*}
\{ d - (b-m) \,:\, \text{$d$ a vertex of $\OG(a,b)_b$ such that $d - (b-m)$ is defined} \}.
\end{equation*}
The edge set of $\OG(a,b)_m$ is given by 
\begin{equation*}
\{ \{d,  d'\} \,:\, 
\text{$\{d + (b-m) , d' + (b-m)\}$ is an edge of $\OG(a,b)_b$} \}.
\end{equation*}
\end{lemma}

In other words, to form $\OG(a,b)_m$ from $\OG(a,b)_{m+1}$, we subtract $1$ from all the boundary points
involved in the vertex diagonals and erase any vertex diagonals (together with their incident edges)
which involve negative boundary points.  This subtraction and deletion process can be 
seen from right to left in Figure~\ref{obstruction}.

\begin{proof} (Sketch.)
The statement about the vertex set of $\OG(a,b)_m$ follows because the set of $a,b$-admissible diagonals
is closed under rotation of $\PP_{b+1}$.  The statement about the edge set of $\OG(a,b)_m$ follows 
by considering `east-west translations' of Dyck paths 
of the form 
\begin{equation*}
N^{i_1} E^{j_1} \dots N^{i_r} E^{j_r} \mapsto N^{i_1} E^{j_1 - 1} N^{i_2} E_{j_2} \dots N^{i_{r-1}} E^{j_{r-1}}
N^{i_r} E^{j_r + 1}.
\end{equation*}
\end{proof}

When 
drawn on $\PP_{b+1}$, the  edges $\{im, km\}$ belonging to $\OG(a,b)_m$ look like wedges whose
common point is $m$ such that $m$ is the largest boundary point on the wedge.  The next two lemmas
concern the more general situation of a wedge $\{im, km\}$ of $a,b$-admissible diagonals with $i < k < m$,
whether or not $\{im, km\}$ is an obstructing edge.  We want to develop some sufficient conditions 
which guarantee that the completion
$ik$ of the triangle is an $a,b$-admissible diagonal.  A schematic of this situation is shown below.

\begin{center}
\includegraphics[scale = 0.5]{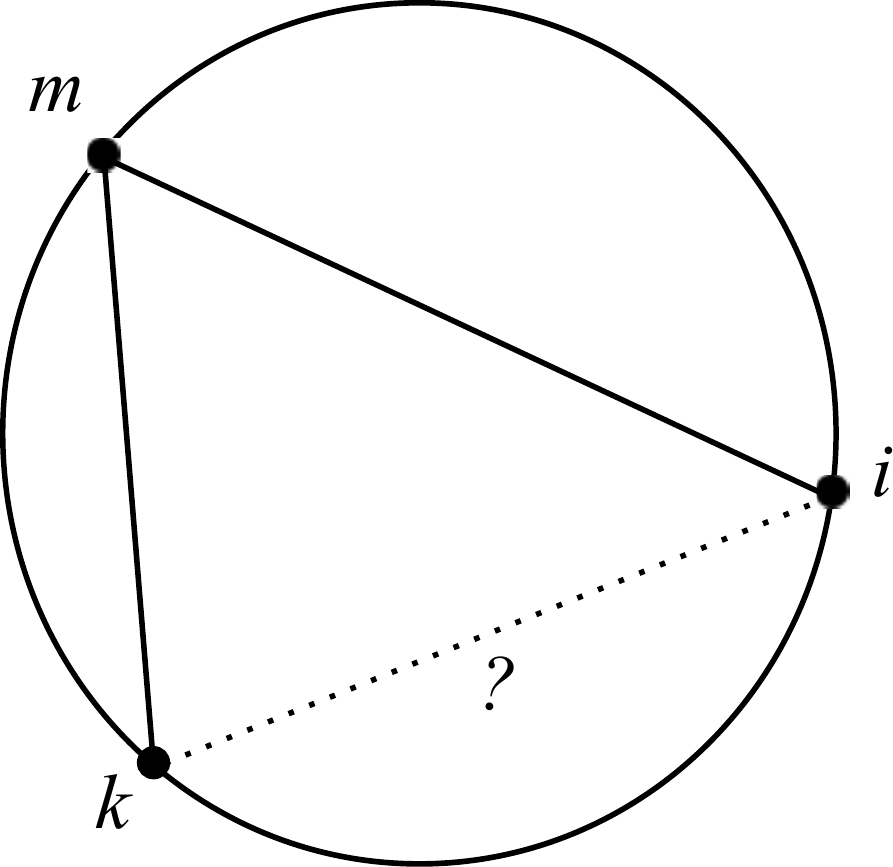}
\end{center}

We are interested in whether $ik$ is $a,b$-admissible because these diagonals
will be used to locate cone vertices as in Lemma~\ref{cone-vertex-lemma}.
It is not always the case that $ik$ is admissible.  For example, if $(a,b) = (5,8)$,
then $\{05, 15\}$ is a wedge of $5,8$-admissible diagonals whose common point is the maximal point but
$01$ is not even a diagonal in $\PP_9$, let alone $5,8$-admissible.  
Our first lemma states that if $\{im, km\}$ is  an edge of $\OG(a,b)$, then $ik$ is an $a,b$-admissible diagonal,
so that the dotted edge is an $a,b$-admissible diagonal in the diagram above.

\begin{lemma}
\label{completing-wedges}
Let $i m$ and $k  m$ be two $a,b$-admissible diagonals 
such that $i < k < m$ and $\{i m, km\}$ is an edge of
$\OG(a,b)$.  We have
that $ik$ is an $a,b$-admissible diagonal.
\end{lemma}

\begin{proof}
We consider the second $a,b$-Dyck path $D'$ drawn in the proof of Lemma~\ref{same-larger-endpoint}.
As was deduced in the proof of Lemma~\ref{same-larger-endpoint}, we know that $R$ and $S$ lie on the
same horizontal step (namely, the horizontal step with right endpoint having $x$-coordinate $m$).  In particular,
the points $R$ and $S$ are $< 1$ unit apart.  This means that the horizontal distance between the lasers
$\ell(P')$ and $\ell(Q')$ is $< 1$ unit.  In particular, the (solid portion of the) laser $\ell(P')$
hits $D'$ at the horizontal step with left endpoint $Q'$.  Since the $x$-coordinate of $P'$ is $i$ and
the $x$-coordinate of $Q'$ is $k$, we conclude that $ik$ is  an $a,b$-admissible diagonal.
\end{proof}

While Lemma~\ref{completing-wedges} will give us some of the cone vertices we will
need to perform our collapsing, we will  need more.  These additional vertices
will be provided by the following result.

\begin{lemma}
\label{completing-half-wedges}
Let $i  m$ and $j  m$ be two $a,b$-admissible diagonals 
 such that $\{i m, j m\}$ is
 an edge of $\OG(a,b)$.  Suppose that 
$i < k < j$ and $k  m$ is an $a,b$-admissible diagonal  such that
$\{k m, j m\}$ is \textbf{not} an edge of $\OG(a,b)$.
Then  $\{i m, km\}$ is an edge of $\OG(a,b)$ and
$i  k$ is an $a,b$-admissible diagonal
\end{lemma}

\begin{center}
\includegraphics[scale = 0.5]{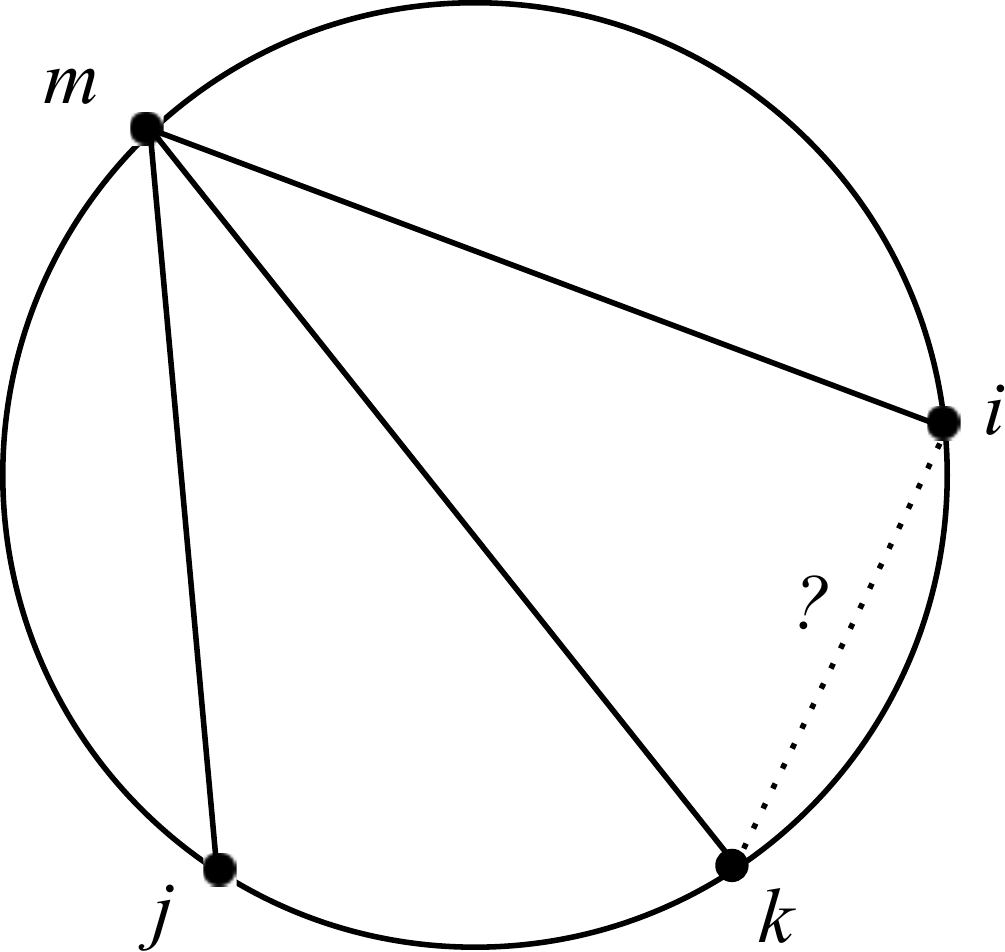}
\end{center}

The situation in Lemma~\ref{completing-half-wedges} is described in the schematic above.
We assume that the big wedge $\{i m, j m\}$ \textbf{is} an edge in $\OG(a,b)$ and that the 
smaller left wedge $\{km, jm\}$ \textbf{is not} an edge in $\OG(a,b)$.   The conclusion is that the dashed segment is an $a,b$-admissible diagonal.

\begin{proof}
The idea is to consider a Dyck path $D$ with valleys at $i, j,$ and $k$.  As in the proof of 
Lemma~\ref{completing-wedges}, we will get that $i k$ appears as a laser diagonal, and is therefore
admissible.
More precisely, let $D$ be the following Dyck path.

\begin{center}
\includegraphics[scale = 0.7]{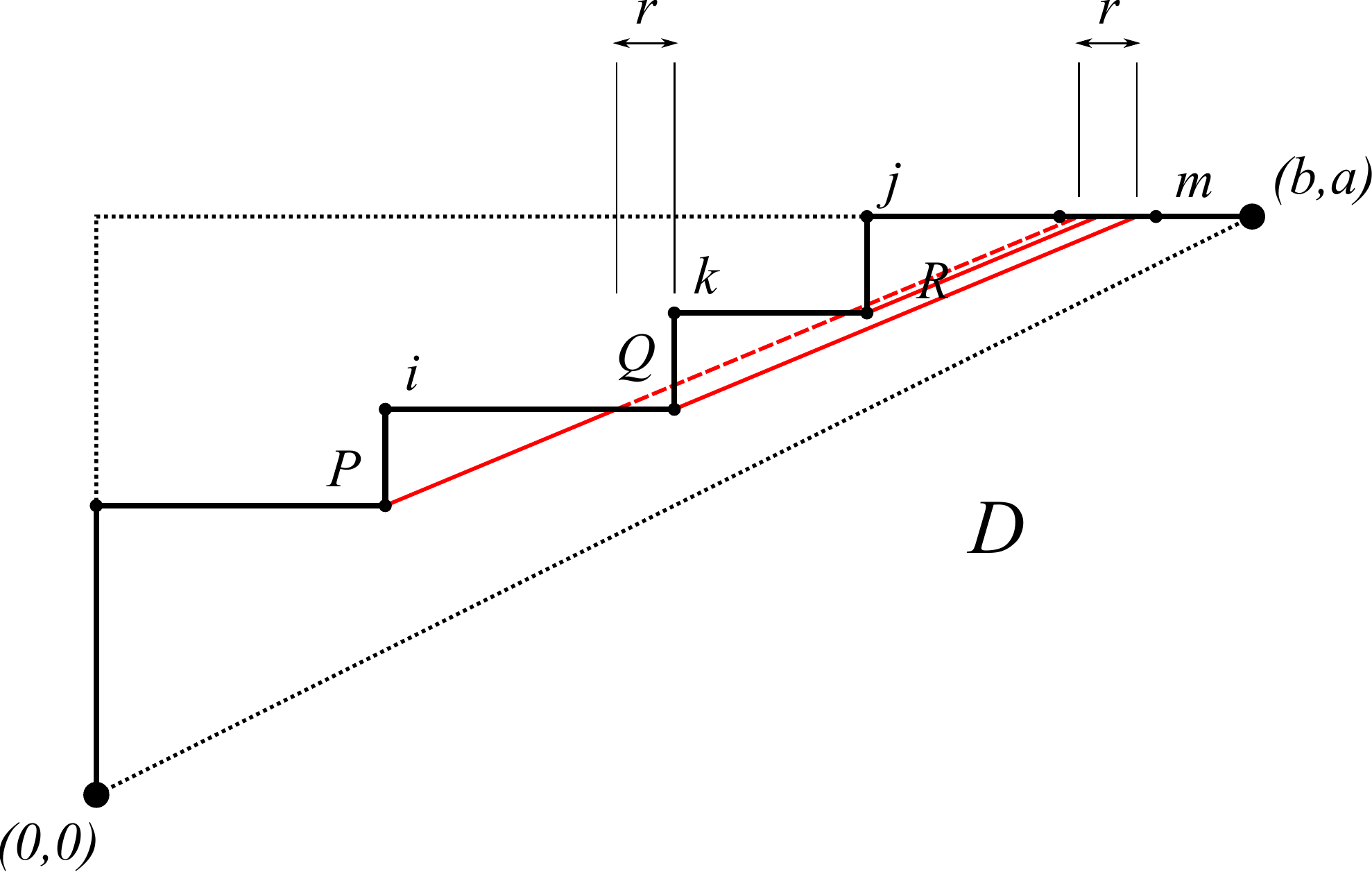}
\end{center}

The valleys $P, Q,$ and $R$ lie on the vertical lines $x = i$, $x = j$, and $x = k$, respectively.  
The point $R$ is chosen so that $d(R) = jm$.  The point $Q$ is chosen so that $d(Q) = km$.  (We know
that $\ell(Q)$ does not hit $D$ in the horizontal run to the left of $R$ because $\{km, jm\}$ is not
an edge of $\OG(a,b)$.)  Finally, the point $P$ is chosen so that the `extended laser' from $P$
(shown here as a dashed line) hits the line $y = a$ at the lattice step whose right endpoint
has $x$-coordinate $m$.  

Since $\{i m, jm \}$ is an edge of $\OG(a,b)$, we know that 
$\ell(P)$ must hit $D$ strictly below the line $y = a$.  Since lasers are parallel, this forces $\ell(P)$ to hit
$D$ on a step of the horizontal run to the left of $Q$.  We claim that the right endpoint of this step is $Q$.
So see this, let $r$ denote the total distance along the line $y = a$ between the `extended' version of the laser
$\ell(P)$ and the laser $\ell(R)$.  Since 
the lasers $\ell(Q)$, $\ell(R)$, and the extended version of $\ell(P)$
intersect the horizontal line $y = a$ at a lattice step whose right
endpoint has $x$-coordinate $m$, we have $r < 1$.  But $r$ is also the distance 
along the horizontal line going through $Q$
between the point where
the unextended laser
$\ell(P)$ hits $D$ and $Q$.  This implies that $\ell(P)$ hits $D$ at the lattice step with left endpoint $Q$
and that $ik$ is $a,b$-admissible. 

To see that $\{i m, km\}$ is an edge of $\OG(a,b)$, consider the path $D'$ whose unique valleys are at 
the lattice points $P$ and $Q$.  
That is, the path $D'$ looks like the path $D$ above with the valley at $R$ `folded out'.
The laser $\ell(P)$ hits $D'$ in the horizontal run 
to the left of $Q$.  By the path building algorithm, we conclude that 
$\{i m, km\}$ is an edge of $\OG(a,b)$.
\end{proof}

\section{Collapsing}
\label{Collapsing}

In this section we describe how to collapse the complex $\Asshat(a,b)$ onto the complex $\Ass(a,b)$.  
By Proposition~\ref{flag}, we know that a face $F$ of $\Asshat(a,b)$ is also a face of $\Ass(a,b)$ if and only 
if $F$ does not contain any edges of the obstruction graph $\OG(a,b)$.  A first approximation of
our collapse is performed by putting a total order on the edges
of $\OG(a,b)$ and repeatedly using Lemma~\ref{cone-vertex-lemma} to collapse all faces 
$F$ of $\Asshat(a,b)$ containing a given edge but none of the previous edges.  However, it will
turn out that some of these `subcollapses' indexed by edges of $\OG(a,b)$ must be further
broken up as a sequence of smaller collapses when the wedges on $\PP_{b+1}$ corresponding to the 
edges of $\OG(a,b)$ are `too wide'.

We begin by describing our total order on the edges of $\OG(a,b)$.  This is essentially an
iterated version of lexicographical order on the $a,b$-admissible diagonals.

\begin{defn} 
Given two $a,b$-admissible diagonals $d = i  j$ and $d' = i'  j'$, we say that $d \sqsubseteq d'$
if $j < j'$ or if $j = j'$ and $i \leq i'$.  Given two edges $e_1 = \{d_1 \sqsubset d'_1\}$ and 
$e_2 = \{d_2 \sqsubset d_2'\}$ of $\OG(a,b)$, we say that $e_1 \preceq e_2$ if  
$d_1 \sqsubset d_2$ or $d_1 = d_2$ and $d'_1 \sqsubseteq d'_2$.
\end{defn}

\begin{example}
\label{orderexample}
If $(a,b) = (5,8)$, the order $\prec$ on the edges of $\OG(a,b)$ is given by
\begin{equation*}
\{04, 24\} \prec \{15, 35\} \prec \{26, 46\} \prec \{07, 27\} \prec \{07, 57\} \prec \{37, 57\} \prec \{18, 38\}
\prec \{18, 68\} \prec \{48, 68\}.
\end{equation*}
\end{example}

Let $\{e_1, \dots, e_N\}$ be the complete set of edges of $\OG(a,b)$ indexed so that
$e_1 \prec \dots \prec e_N$.
The order $\preceq$ can be used to define a family
of simplicial complexes.

\begin{defn}
For $0 \leq r \leq N$, inductively define a family of complexes $\Ass(a,b)_r$ by
$\Ass(a,b)_N := \Asshat(a,b)$ and
$\Ass(a,b)_r := \dl_{\Ass(a,b)_{r+1}}(e_{r+1})$ for $0 \leq r \leq N-1$.
\end{defn}

In other words, we have that 
\begin{equation*}
\Ass(a,b)_r = \{F \in \Asshat(a,b) \,:\, \text{$F$ does not contain any edges $e_s$ of $\OG(a,b)$ for $s > r$} \}
\end{equation*}
We are  ready to prove Theorem~\ref{main}.

\begin{proof} (of Theorem~\ref{main})
By the definition of the complexes $\Ass(a,b)_r$, we have that $\Ass(a,b)_0 = \Ass(a,b)$ and
$\Ass(a,b)_N = \Asshat(a,b)$.  Fix $1 \leq r \leq N$.  To prove that $\Asshat(a,b)$ collapses
onto $\Ass(a,b)$ it is enough to show that $\Ass(a,b)_r$ collapses onto $\Ass(a,b)_{r-1}$.

Let $e_r = \{d \sqsubset d'\}$, where $d = ik$ and $d' = jk$ are the two $a,b$-admissible diagonals
in $\PP_{b+1}$ which form the edge $e_r$.  By the definition of $\sqsubset$ we have $i < j < k$.

Since $e_r$ does not contain any other edges of $\OG(a,b)$, we know that $e_r$ is a face of
$\OG(a,b)_r$.  Moreover, Lemma~\ref{completing-wedges} guarantees that the segment 
$ij$ is an $a,b$-admissible diagonal, so that
$e_r \cup \{ij\} = \{ij, ik, jk\}$ is a face of $\Asshat(a,b)$.  In fact, since the only edge of 
$\OG(a,b)$ contained in $e_r \cup \{ij\}$ is $e_r$, we have that
$e_r \cup \{ij\}$ is a face of $\Ass(a,b)_r$.  It is tempting to hope that 
$ij$ is a cone vertex for $e_r$ in $\Ass(a,b)_r$, so that we could apply Lemma~\ref{cone-vertex-lemma}
to collapse $\Ass(a,b)_r$ onto $\dl_{\Ass(a,b)_r}(e_r) = \Ass(a,b)_{r-1}$.  Unfortunately, the diagonal
$ij$ is not necessarily a cone vertex for $e_r$ in $\Ass(a,b)_r$ because there may exist faces
$F$ of $\Ass(a,b)_r$ containing $e_r$ which contain diagonals which cross $ij$.  To get around
this problem, we will have to be more delicate and realize the collapsing of
$\Ass(a,b)_r$ onto $\Ass(a,b)_{r-1}$ as a sequence of smaller collapsings indexed by a carefully
chosen sequence of these `crossing faces' $F$.

Let $F$ be a face of $\Ass(a,b)_r$ containing $e_r$ and suppose that $F$ contains an $a,b$-admissible
diagonal $s m$ which crosses the diagonal $ij$.  Since $F$ contains the wedge
$\{ik, jk\}$ and the diagonals in $F$ do not cross, we have that $m = k$ and $i < s < j$.
If $\{sk, jk\}$ were an edge of $\OG(a,b)$, we would have that
$\{ik, jk\} \prec \{sk, jk\} \subset F$, which contradicts the assumption that
$F$ is a face of $\Ass(a,b)_r$.  We conclude that $\{sk, jk\}$ is not an edge of $\OG(a,b)$.

Let $s_1 < s_2 < \dots < s_p$ be a complete list of the indices $i < s_q < j$ such that
\begin{itemize}
\item the segment $s_q k$ is an $a,b$-admissible diagonal and
\item the pair $\{s_q k, jk\}$ is not an edge of $\OG(a,b)$.
\end{itemize}
It is possible that $p = 0$, in which case the rest of the argument simplifies.  For $1 \leq q \leq p$,
define a face $F_q$ of $\Asshat(a,b)$ by
$F_q = \{ik, s_q k, jk\}$.  For all $1 \leq q \leq p$ the only edge of $\OG(a,b)$
contained in $F_q$ is $e_r$, so $F_q$ is a face of $\Ass(a,b)_r$.
For $1 \leq q \leq p+1$, we define a subcomplex $\Ass(a,b)^{(q)}_r$ of $\Ass(a,b)_r$ by
\begin{align*}
\Ass(a,b)_r^{(q)} &:= \dl_{\Ass(a,b)_r}( \{F_1, F_2, \dots, F_{q-1}\} )  \\ &= 
\{ F \in \Ass(a,b)_r \,:\, \text{$F$ does not contain $F_{q'}$ for any $q' < q$} \}.
\end{align*}
In particular, we have that $\Ass(a,b)_r^{(1)} = \Ass(a,b)_r$ and that any face of
$\Ass(a,b)_r^{(p+1)}$ containing $e_r$ does not contain any of the diagonals 
$s_q k$ for $1 \leq q \leq p$.  Moreover, we have that
$\Ass(a,b)_r^{(q+1)} = \dl_{\Ass(a,b)_r^{(q)}}(F_q)$ for $1 \leq q \leq p$.

Fix $1 \leq q \leq p$.  We claim that $\Ass(a,b)_r^{(q)}$ collapses onto $\Ass(a,b)_r^{(q+1)}$.  
To see this, observe that $F_q = \{ik, s_q k, jk\}$, where $i < s_q < j < k$,
$\{ik, jk\}$ is an edge of $\OG(a,b)$, and
$\{s_q k, jk\}$ is not an edge of $\OG(a,b)$.   Applying Lemma~\ref{completing-half-wedges},
we get that $i s_q$ is  an $a,b$-admissible diagonal in $\PP_{b+1}$.

We claim that the $a,b$-admissible diagonal $i s_q$ is a cone vertex for 
$F_q$ in the complex $\Ass(a,b)_r^{(q)}$.  Since $i s_q \notin F_q$, by the definition of cone 
vertices we need to show that if $F$ is any face of $\Ass(a,b)_r^{(q)}$ containing $F_q$,
we have that $F \cup \{i s_q\}$ is also a face of $\Ass(a,b)_r^{(q)}$.

We begin by showing that $F \cup \{i s_q\}$ is a face of $\Asshat(a,b)$, that is, that the diagonals
in $F \cup \{i s_q\}$ are noncrossing on $\PP_{b+1}$.  Indeed, since $F$ contains
$F_q = \{ik, i s_q, jk\}$, the only $a,b$-admissible diagonals in $F$ which could cross $i s_q$
would be of the form $i s_{q'}$ for $q' < q$.  If any of these diagonals were contained in $F$,
then $F$ would contain $F_{q'} = \{ik, i s_{q'}, jk\}$, which contradicts the assumption that
$F$ is a face of $\Ass(a,b)_r^{(q)}$.  We conclude that the diagonals in $F \cup \{i s_q\}$ are
noncrossing and that $F \cup \{i s_q\}$ is a face of $\Asshat(a,b)$.

We still need to show that $F \cup \{i s_q\}$ is a face of the subcomplex
$\Ass(a,b)_r^{(q)}$ of $\Asshat(a,b)$.  This amounts to showing that $F \cup \{i s_q\}$ does not 
contain 
\begin{enumerate}
\item any edge $e_{r'}$ of $\OG(a,b)$ with $r' > r$, or  
\item any face $F_{q'}$ with
$q' < q$.  
\end{enumerate}
For (1), since $F$ is a face of $\Ass(a,b)_r^{(q)}$, we know that $F$ does not
contain any edge $e_{r'}$ of $\OG(a,b)$ with $r' > r$.  What possible edges of 
$\OG(a,b)$ could be added to $F$ by introducing the diagonal $i s_q$?  Only edges
of the form $e_{r'} = \{i' s_q, i s_q\}$ which belong to the component
$\OG(a,b)_{s_q}$.
But by the definition of $\prec$ and the fact that 
$s_q < k$, we know that any such edge would satisfy $e_{r'} \prec e_r$, so $r' < r$.
This implies (1).  For (2), we make the simple observation that $i s_q$ is not 
of the form $s_{q'} k$, so $F \cup \{i s_q\}$ does not contain $F_{q'}$ for $q' < q$ because $F$
does not contain $F_{q'}$ for $q' < q$.

By the last two paragraphs, we have that $i s_q$ is a cone vertex for $F_q$ in
$\Ass(a,b)^{(q)}_r$.
By Lemma~\ref{cone-vertex-lemma} and induction, we get that 
$\Ass(a,b)_r = \Ass(a,b)^{(1)}$ collapses onto $\Ass(a,b)_r^{(p+1)}$.  To finish the proof of
Theorem~\ref{main}, we will show that $\Ass(a,b)_r^{(p+1)}$ collapses onto
$\Ass(a,b)_{r-1}$.

We begin with the observation that $e_r$ is a face of $\Ass(a,b)_r^{(p+1)}$ and that we have
$\Ass(a,b)_{r-1} = \dl_{\Ass(a,b)_r^{(p+1)}}(e_r)$.  By 
Lemma~\ref{completing-wedges} we get that $ij$ is an $a,b$-admissible diagonal in $\PP_{b+1}$.
We claim that $ij$ is a cone vertex for $e_r$ in $\Ass(a,b)_r^{(p+1)}$.
To see this, let $F$ be any face of $\Ass(a,b)_r^{(p+1)}$ containing $e_r$.  Since
$ij \notin e_r$, it suffices to show that
$F \cup \{ij\}$ is also a face of $\Ass(a,b)_r^{(p+1)}$.

As before, we begin by showing that $F \cup \{ij\}$ is a face of $\Asshat(a,b)$, i.e., that
the diagonals in $F \cup \{ij\}$ do not cross.  Indeed, since $F$ contains the wedge
$e_r = \{ik, jk\}$, the only diagonals in $F$ which could cross $ij$ would be 
$a,b$-admissible diagonals of the form $sk$ for $i < s < j$.  But since $F$ is a face 
of $\Ass(a,b)_r^{(p+1)} \subseteq \Ass(a,b)_r$, the definition of $\prec$ would force
the pair $\{sk, jk\}$ to be a non-edge in $\OG(a,b)$.  This means that $s = s_q$ for some 
$1 \leq q \leq p$ and $F$ would contain
$F_q = \{ik, s_qk, jk\}$.  But this contradicts that fact that
$F$ is a face of $\Ass(a,b)_r^{(p+1)}$.  We conclude that none of the diagonals in $F$
cross $ij$ and $F \cup \{ij\}$ is a face of $\Asshat(a,b)$.

We still need to show that $F \cup \{ij\}$ is a face of the subcomplex
$\Ass(a,b)_r^{(p+1)}$ of $\Asshat(a,b)$.  This amounts to showing that 
$F \cup \{ij\}$ does not contain 
\begin{enumerate}
\item
any edge $e_{r'}$ of $\OG(a,b)$ with $r' > r$, or
\item any face $F_q$ for $1 \leq q \leq p$. 
\end{enumerate} 
For (1), observe that the addition of $ij$ to $F$ can only 
add edges of $\OG(a,b)$ of the form $\{i'j, ij\}$
which belong to the component $\OG(a,b)_j$ and since $j < k$
any such edge is $\prec e_r$.  For (2),
observe that $ij \notin F_q$ for $1 \leq q \leq p$.

By the last two paragraphs, we have that $ij$ is a cone vertex for $e_r$ in
$\Ass(a,b)_r^{(p+1)}$.  By Lemma~\ref{cone-vertex-lemma}, we have that
$\Ass(a,b)_r^{(p+1)}$ collapses onto $\Ass(a,b)_{r-1}$.  We already showed that
$\Ass(a,b)_r$ collapses onto $\Ass(a,b)_r^{(p+1)}$, so we have the overall 
collapse of $\Ass(a,b)_r$ onto $\Ass(a,b)_{r-1}$.  By induction, we have that
$\Asshat(a,b) = \Ass(a,b)_N$ collapses onto $\Ass(a,b) = \Ass(a,b)_0$.  This completes
the proof of Theorem~\ref{main}.
\end{proof}

\begin{example}
\label{collapseexample}
To better understand the argument used in the proof of Theorem~\ref{main}, we show how to collapse
$\Asshat(5,8)$ onto $\Ass(5,8)$.

The obstruction graph $\OG(5,8)$ has $N = 9$ edges and the order presented in 
Example~\ref{orderexample} labels these edges as $e_1 \prec \dots \prec e_9$.  We have that
$\Ass(5,8)_9 = \Asshat(5,8)$ and $\Ass(5,8)_0 = \Ass(5,8)$.  

To show that $\Ass(5,8)_9$ collapses onto $\Ass(5,8)_8$, we use the following diagram.
\begin{center}
\includegraphics[scale = 0.5]{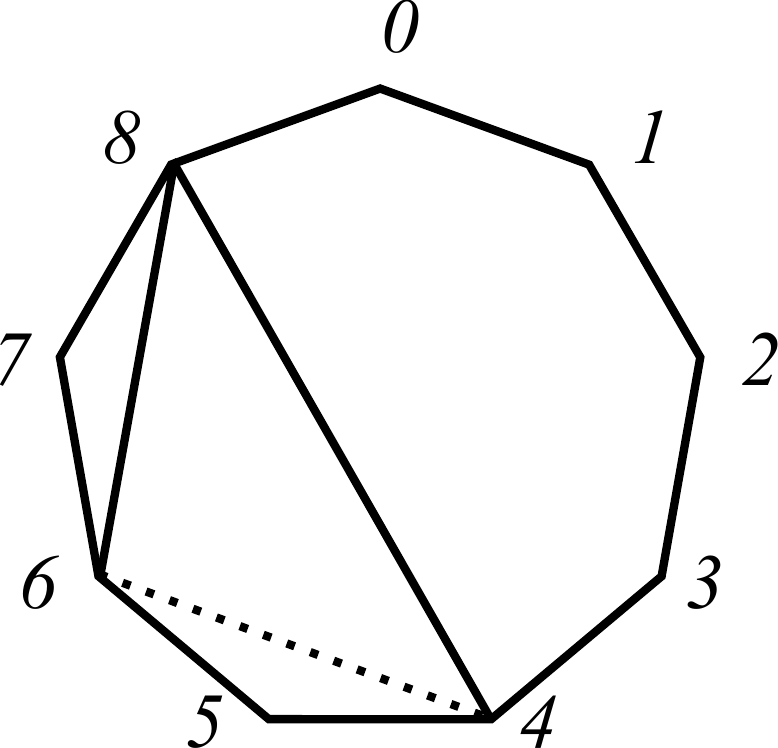}
\end{center}
The diagonals in $e_9 = \{48, 68\}$ are solid.  The dotted diagonal $46$ (which is guaranteed
to be $5,8$-admissible by Lemma~\ref{completing-wedges}) is a cone vertex for 
$e_9$ in $\Ass(5,8)_9$ because no face of $\Ass(5,8)_9$ containing $e_9$ contains a diagonal
which crosses $46$.  Applying Lemma~\ref{cone-vertex-lemma}, we see that 
$\Ass(5,8)_9$ collapses onto $\Ass(5,8)_8$.

The argument that $\Ass(5,8)_8$ collapses onto $\Ass(5,8)_7$ is more complicated.  If we try to use 
the above strategy directly, we encounter the following diagram.
\begin{center}
\includegraphics[scale = 0.5]{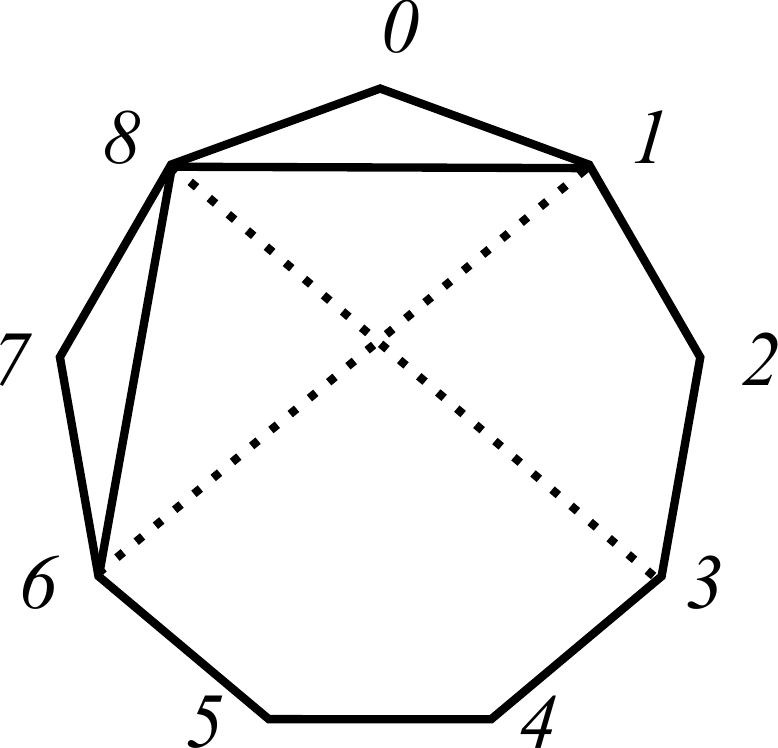}
\end{center}
The diagonals in $e_8 = \{18, 68\}$ are solid.  Lemma~\ref{completing-wedges} guarantees that
the dotted diagonal $16$ is $5,8$-admissible.  However, the diagonal $16$ is not a cone vertex
for $e_8$ in $\Ass(5,8)_8$ because it crosses the diagonal $38$.  (The diagonal $48$ is not 
a problem because any face of $\Ass(5,8)_8$ containing $e_8$ cannot contain $48$, as such
a face would then contain $e_9$.  In other words, the diagonal $48$ forms an obstructing
edge with the diagonal $68$.)  

To get around this problem, we define a face $F_1$ of $\Ass(5,8)_8$ by 
$F_1 := \{18, 38, 68\}$.  We have the following diagram.
\begin{center}
\includegraphics[scale = 0.5]{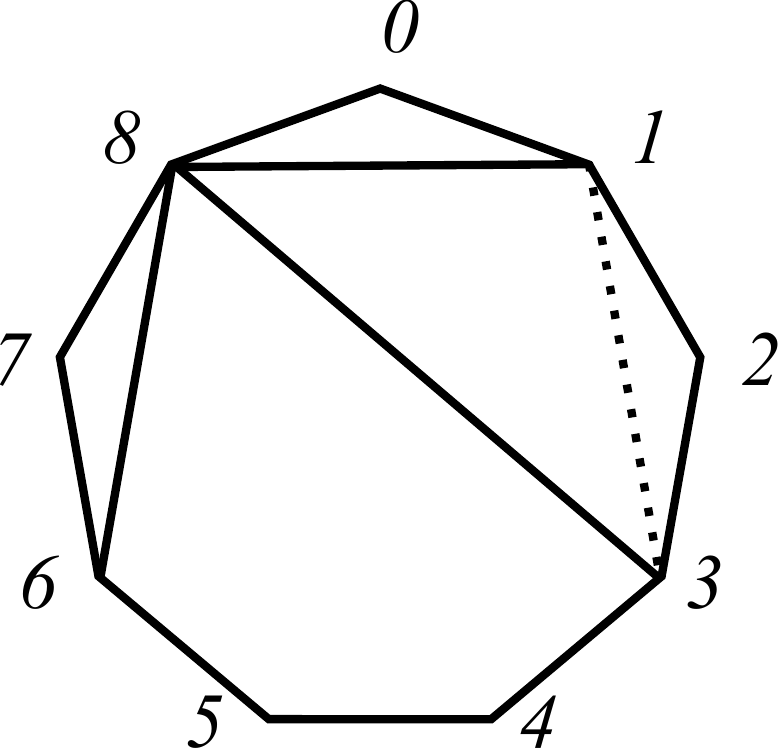}
\end{center}
The diagonals in $F_1$ are solid.  Lemma~\ref{completing-half-wedges} guarantees that 
the dotted diagonal $13$ is $5,8$-admissible and it is the case that $13$ is a cone vertex for $F_1$ in
$\Ass(5,8)_8$.  By Lemma~\ref{cone-vertex-lemma}, the complex $\Ass(5,8)_8 = \Ass(5,8)_8^{(1)}$ collapses 
onto the complex $\dl_{\Ass(5,8)_8}(F_1) = \Ass(5,8)_8^{(2)}$.

Working inside the complex $\Ass(5,8)_8^{(2)}$, we can revert to our previous argument.  We have the following
diagram.
\begin{center}
\includegraphics[scale = 0.5]{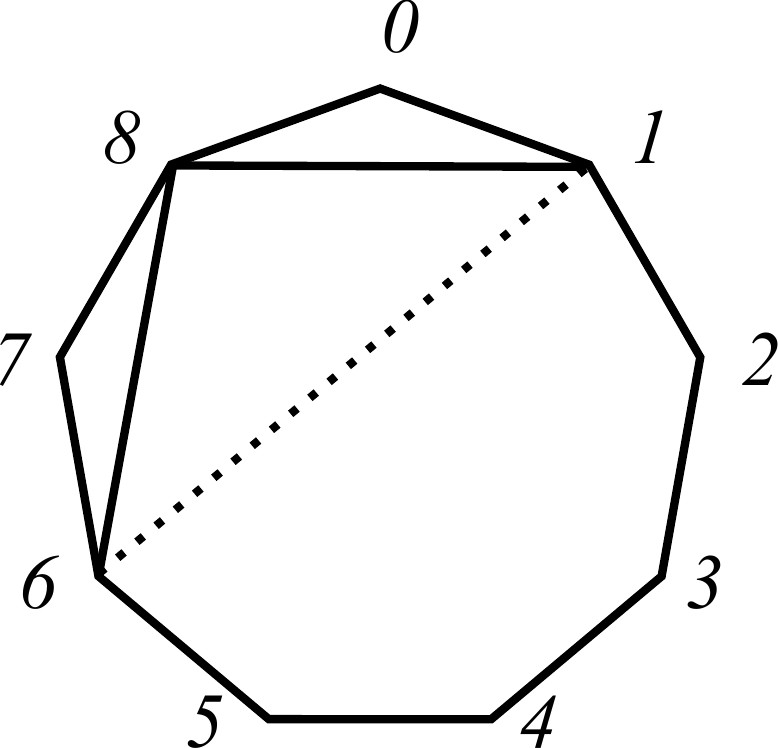}
\end{center}
The diagonal $38$ is no longer present because no face of $\Ass(5,8)_8^{(2)}$ contains $F_1$.  We can apply
Lemmas~\ref{completing-wedges} to conclude that  $16$ is $5,8$-admissible and
observe that $16$ is a cone vertex
for $e_8$ in $\Ass(5,8)_8^{(2)}$.  We conclude that $\Ass(5,8)_8^{(2)}$ collapses onto $\Ass(5,8)_7$.

The proof that $\Ass(5,8)_7$ collapses onto $\Ass(5,8)_6$ is easier and relies on the following diagram.
\begin{center}
\includegraphics[scale = 0.5]{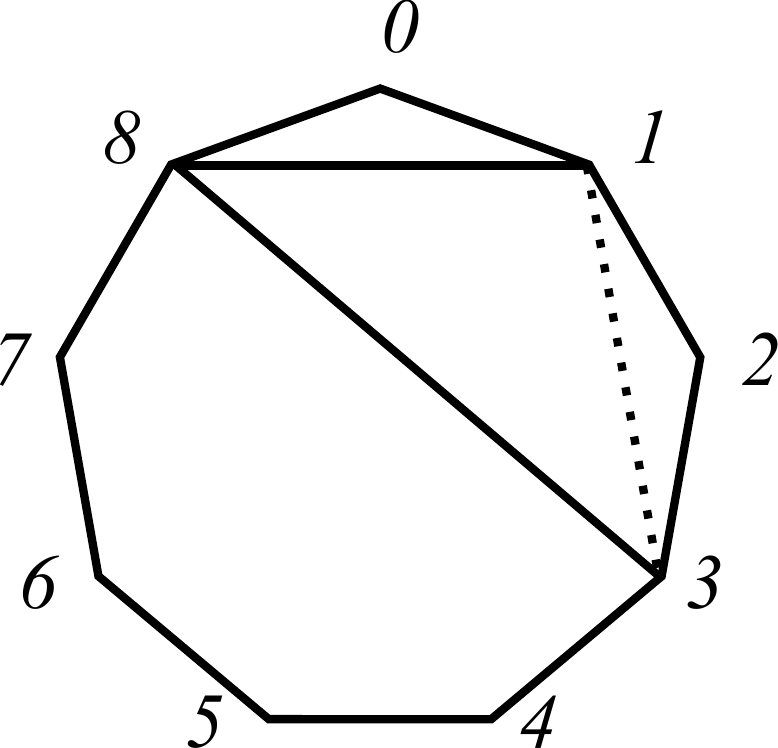}
\end{center}
Since the obstructing edge $e_7$ forms a sufficiently narrow wedge $\{18, 38\}$ in $\PP_9$, there is no room
for a $5,8$-admissible diagonal to get in the way of $13$ being a cone vertex.

At this point, we have collapsed all the faces of $\Asshat(5,8)$ which contain an obstructing edge involving
the wedge boundary point $8$.  We move along to those obstructing edges involving the wedge 
boundary point $7$.  
This process begins by considering the following diagram to give the collapse of $\Ass(5,8)_6$ onto
$\Ass(5,8)_5$.
\begin{center}
\includegraphics[scale = 0.5]{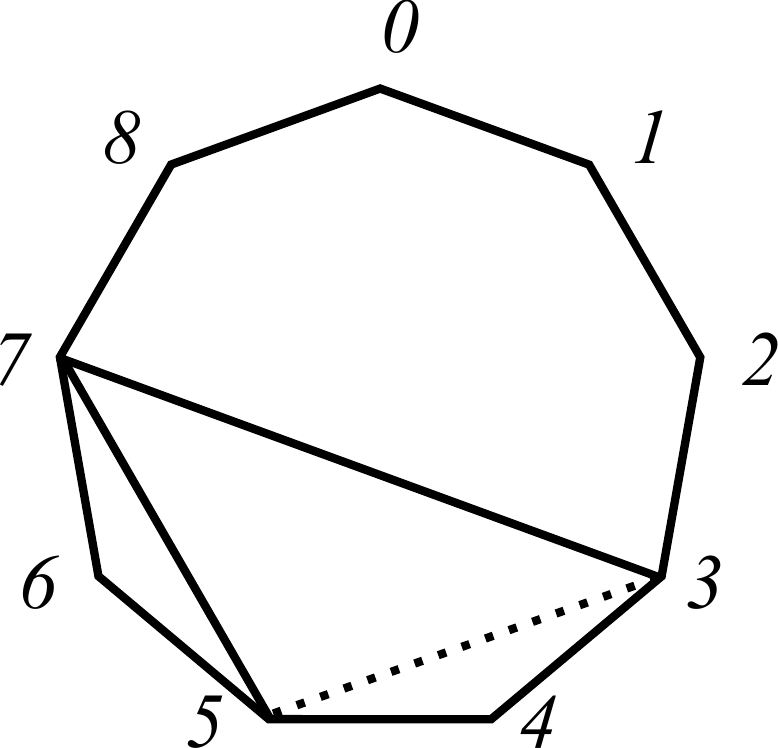}
\end{center}
The dotted vertex $35$ plays the role of the cone vertex and we repeat the sequence of arguments above,
rotated one boundary point counterclockwise.  
The collapses corresponding to the obstructing edges with wedge boundary points $6,5,4$ are easier
and complete the collapsing of $\Asshat(5,8)$ onto $\Ass(5,8)$.
Observe that the dotted diagonals considered when collapsing
from $\Ass(5,8)_r$ to $\Ass(5,8)_{r-1}$ never create an obstructing edge with index $> r$, so really
are cone vertices inside the complex $\Ass(5,8)_r$.
\end{example}

We reiterate the deduction of Corollary~\ref{alexander} from Theorem~\ref{main}
presented in \cite{ARW}.

\begin{proof} (of Corollary~\ref{alexander})
An argument using the coprimality of $a$ and $b$ and the sets $S(a,b)$ and $S(b-a,b)$ shows that 
any diagonal of $\PP_{b+1}$ is either $a,b$-admissible or $(b-a,b)$-admissible, but not both.
In other words, the vertex sets of $\Asshat(a,b)$ and $\Asshat(b-a,b)$ partition the 
vertex set of the simplicial sphere $\Ass(b-1,b)$.  By the definition of $\Asshat$, a subset $F$
of $\Asshat(a,b)$ is a face of $\Asshat(a,b)$ if and only if $F$ is also a face of 
$\Ass(b-1,b)$ and similarly for $\Asshat(b-a,b)$.  This means that the complement
of $\Asshat(a,b)$ within $\Ass(b-1,b)$ deformation retracts onto $\Asshat(b-a,b)$. 
In other words, we have that $\Asshat(a,b)$ and $\Asshat(b-a,b)$ are Alexander dual
within the sphere $\Ass(b-1,b)$.  By Theorem~\ref{main}, we also have that
$\Ass(a,b)$ and $\Ass(b-a,b)$  are  Alexander dual within the sphere $\Ass(b-1,b)$. 
\end{proof}

\section{Closing remarks}
\label{Open}

The construction $(a, b) \leadsto \Ass(a,b)$ attaches a simplicial complex to every pair of coprime
positive integers $a < b$.
In this paper we proved that the duality 
$(a, b) \longleftrightarrow (b-a,b)$ on pairs of coprime integers manifests itself topologically 
as Alexander duality of simplicial complexes.  

Any increasing pair of coprime positive integers an be obtained from $(1, 2)$ by repeated applications
of the duality $(a,b) \mapsto (b-a, b)$ and the map $(a, b) \mapsto (b, 2b-a)$.  Alternatively, we could
use repeated applications of the duality $(a,b) \mapsto (b-a,b)$ and the map 
$(a,b) \mapsto (a, a+b)$.  It may be interesting to find a topological manifestation 
of the maps $(a,b) \mapsto (b, 2b-a)$ and $(a,b) \mapsto (a, a+b)$ in terms of rational 
associahedra.  These maps are closely related to the the monoid endomorphisms
of the set $\{x, y\}^*$ of words on the two-letter alphabet $\{x, y\}$ which generate the 
so-called `Christoffel morphisms' (see \cite{BLRS}).  In the notation of \cite{BLRS}, the 
duality $(a,b) \mapsto (b-a,b)$ corresponds to the morphism \textbf{E}, the 
map $(a,b) \mapsto (b, 2b-a)$ corresponds to the morphism $\widetilde{\textbf{D}}$, and the map
$(a,b) \mapsto (a, a+b)$ corresponds to the morphism \textbf{G}.  Since Christoffel words can be
identified with a subset of
rational Dyck paths, it may be interesting to develop a deeper connection between Christoffel theory
and rational associahedra.

The main topological tool used in proving Theorem~\ref{main} was the elementary result of
Lemma~\ref{cone-vertex-lemma}.  Discrete Morse Theory was introduced by Forman \cite{Forman}
and is a powerful and ubiquitous tool in geometric combinatorics
for proving deformation retraction type results.  More precisely, consider the 
\textbf{\textit{face poset}} $\widehat{P}$ of $\Asshat(a,b)$ (that is, the set of faces of $\Asshat(a,b)$ partially ordered
by inclusion).  The poset $\widehat{P}$ contains the face poset $P$ of $\Ass(a,b)$ as an order ideal.
The collapse of $\Asshat(a,b)$ onto $\Ass(a,b)$ in Theorem~\ref{main} gives a perfect matching
on the Hasse diagram of the difference $\widehat{P} - P$.
It follows that this perfect matching is Morse (see \cite[Chapter 11]{Koslov} for the definition of 
a Morse matching).  Conversely, any Morse matching on $\widehat{P} - P$ gives rise to a collapse
of $\Asshat(a,b)$ onto $\Ass(a,b)$.  It may be interesting to prove Theorem~\ref{main} by exhibiting a 
Morse matching on $\widehat{P} - P$.

It remains an open problem to extend most of the constructions of rational Catalan theory to reflection
groups $W$ other than the symmetric group $\symm_n$.  The $W$-analog of the associahedron 
is given by the \textbf{\textit{cluster complexes}} of Fomin and Zelevinsky \cite{FZ} and a 
Fu$\ss$ analog of this construction is given by the \textbf{\textit{generalized cluster complexes}} of 
Fomin and Reading \cite{FR}.  A rational analog of cluster complexes would have as its input data a pair
$(W, b)$, where $W$ is an irreducible
reflection group with Coxeter number $h$ and $b > h$ is an integer coprime to $h$.
The `W-rational associahedron' $\Ass(W,b)$ should 
be pure of dimension $h-2$, have $\Cat(W, b) := \prod_{i = 1}^n \frac{e_i + b}{e_i + 1}$ facets (where
$e_1, \dots, e_n$ are the exponents of $W$), and have $f$- and $h$-vector entries given by
rational analogs of the $W$-Kirkman and Narayana numbers.  

Unfortunately, it does not seem possible to extend the Alexander duality proven in this paper to the 
full generality of `type $W$'.  Such a duality would exist between $\Ass(W,b)$ and $\Ass(W',b)$, where
$W'$ is a reflection group with Coxeter number $b-h$.  Indeed, within the BCD infinite series, the only possible
choices for $h$ are even, forcing both $b$ and $b-h$ to be odd.

\section{Acknowledgements}
\label{Acknowledgements}

The author thanks Drew Armstrong, Benjamin Braun, Steven Klee, and Nathan Williams for helpful conversations.
The author was partially supported by NSF grant DMS - 1068861.

\end{document}